\documentclass[12pt]{amsart}
\usepackage{amssymb,amsmath,amsthm,amscd}
\usepackage{slashed}
\usepackage{hyperref}
\usepackage[alphabetic, msc-links]{amsrefs}
\hypersetup{
  colorlinks   = true,	%Colour links instead of ugly boxes
  urlcolor     = blue,	%Colour for external hyperlinks
  linkcolor    = blue,	%Colour of internal links
  citecolor   = red	%Colour of citations
}
\calclayout
\newtheorem{theorem}{Theorem}[section]
\newtheorem{proposition}[theorem]{Proposition}
\newtheorem{corollary}[theorem]{Corollary}
\newtheorem{lemma}[theorem]{Lemma}
\newtheorem{example}[theorem]{Example}
\newtheorem{definition}[theorem]{Definition}
\newtheorem{remark}[theorem]{Remark}
\numberwithin{theorem}{section} \numberwithin{equation}{section}
\newcommand{\beq}{\begin{eqnarray*}}
\newcommand{\eeq}{\end{eqnarray*}}
\newcommand{\beqn}{\begin{eqnarray}}
\newcommand{\eeqn}{\end{eqnarray}}
\newcommand{\overbar}[1]{\mkern 1.5mu\overline{\mkern-1.5mu#1\mkern-1.5mu}\mkern 1.5mu}
\newcommand{\del}{\partial}
\newcommand{\delbar}{\bar{\partial}}
\newcommand{\re}{\operatorname{Re}}
\newcommand{\im}{\operatorname{Im}}
\title{From the signature theorem to anomaly cancellation}
\author{Andreas Malmendier}
\address{Department of Mathematics and Statistics, Utah State University, Logan, UT 84322}
\email{andreas.malmendier@usu.edu}
\thanks{The first author acknowledges support from the Simons Foundation through grant no.~202367.}
\author{Michael T. Schultz}
\address{Department of Mathematics and Statistics, Utah State University, Logan, UT 84322}
\email{michael.schultz@usu.edu}
\thanks{The second author acknowledges support from a Research Catalyst grant by the Office of Research and Graduate Studies at Utah State University.}
\begin{document}
\begin{abstract}
We survey the Hirzebruch signature theorem as a special case of the Atiyah-Singer index theorem. The family version of the Atiyah-Singer index theorem in the form of the Riemann-Roch-Grothendieck-Quillen (RRGQ) formula is then applied to the complexified signature operators varying along the universal family of elliptic curves. The RRGQ formula allows us to determine a generalized cohomology class on the base of the elliptic fibration that is known in physics as (a measure of) the local and global anomaly. Combining several anomalous operators allows us to cancel the local anomaly on a Jacobian elliptic surface, a construction that is based on the construction of the Poincar\'e line bundle over an elliptic surface.
\end{abstract}
\maketitle
\section{Introduction}
A genus is a ring homomorphism from Thom's oriented cobordism ring to
the complex numbers. It is known that the oriented cobordism ring tensored
with the rational numbers is a polynomial ring generated by the even
dimensional complex projective spaces. To give a genus is therefore
the same thing as to give its evaluation on all even dimensional
complex projective spaces. Moreover, any genus can be expressed
uniquely as the evaluation of a stable exponential characteristic
class; in other words, a polynomial in the Pontrjagin classes of the tangent
bundle evaluated on the fundamental class of the manifold.
\par The signature is a ring homomorphism defined as the index of the
intersection form on the middle cohomology group. Equivalently, there is a 
topological index, the Hirzebruch $L$-genus of a manifold. Here, the topological index equals
the analytic index of the signature operator, a special kind of chiral Dirac operator. 
The \emph{Hirzebruch signature theorem} asserts that the topological index equals 
the analytic index and implies the Hirzebruch-Riemann-Roch theorem, the first successful 
generalization of the classical Riemann-Roch theorem for complex algebraic varieties of all dimensions \cites{MR0063670, MR0074086}.
These theorems are special cases of the Atiyah-Singer index theorem \cite{MR157392} for suitably defined signature operators. 
\par When one considers a \emph{family} of complexified signature operators $\{\slashed{D}_g\}$ -- where we denote the operators as chiral Dirac operators as in physics -- acting on the sections of suitable bundles over a fixed even-dimensional manifold $M$, with the operators being parameterized by (conformal classes of) Riemannian metrics $g$, one can generalize the meaning of index and use the family version of the Atiyah-Singer index theorem \cite{MR0279833}. We will focus on the simplest nontrivial case, when $M$ is a flat two-torus, and $g$ varies over the moduli space $\mathfrak{M}$ of flat metrics. Given $g \in \mathfrak{M}$, we are motivated by the following question: 
\begin{flushleft}
\textbf{Question:} \emph{Is there any anomalous behavior of the family of chiral Dirac operators 
\beq 
 \slashed{D}_g :  C_+^\infty(M, \Lambda^*) \to C_-^\infty(M, \Lambda^*)\,,
 \eeq 
acting on complex-valued differential forms on $M$ as $g \in \mathfrak{M}$ varies?}
\end{flushleft}
\par Instead of looking at conformal Riemannian structures on the two-torus $M$, we identify $M=\mathbb{C} / \langle 1, \tau\rangle$ as quotient of the complex plane by the rank-two lattice $\langle 1, \tau\rangle$ in $\mathbb{C}$. This is done by identifying opposite edges of each parallelogram spanned by $1$ and $\tau$, with $\im(\tau)>0$, in the lattice to obtain $M=\mathbb{C} / \langle 1, \tau\rangle$. We then endow $M$ with a compatible flat torus metric $g$ that descends from the flat metric on $\mathbb{C}$ and think of $\tau \in \mathbb{H}$  as the complex structure on $M$. Here $\mathbb{H} \subset \mathbb{C}$ is the upper half plane. Then the moduli space of complex structures on $M$ is given by $\mathfrak{M}=\mathbb{H}/\operatorname{PSL}(2,\mathbb{Z})$; its compactification is isomorphic to $\mathbb{C} P^1$ and is called the $j$-line. Thus, we identify the conformal class of a flat torus metric $g$ with the isomorphism class $\tau$ of the complex structure on $M$. As one varies $\tau$ over $\mathbb{C} P^1$, one obtains a rational Jacobian elliptic surface, called the universal family of elliptic curves, given as a holomorphic family of elliptic curves  (some of them singular) over the $j$-line. The relationship between $j$ and $\tau$ is clarified in Section \ref{section:ellipticsurfaces}.
\par The Hirzebruch signature theorem asserts that any $n$-manifold $M$ arising as the boundary of $(n+1)$-manifold has vanishing signature. Thus, the numerical index of every chiral Dirac operator $\slashed{D}_\tau$ in each smooth torus fiber of the aforementioned Jacobian elliptic surface vanishes since the signature of a smooth torus is zero. However, there are points in the moduli space where the symmetry group of the elliptic curve representing the torus fiber jumps from $\mathbb{Z}_2$ to another discrete group \cite{MR0344216}. It is this sudden change in the holomorphic structure that gives rise to a so-called \emph{anomaly}.
\par To detect this anomaly, we are interested in Quillen's \emph{determinant line bundle} \cite{MR783704} $\operatorname{\mathbf{Det}}{\slashed{D}} \to \mathbb{C} P^1$ associated with the family of Dirac operators $\slashed{D}_\tau$, and its first Chern class $c_1(\operatorname{\mathbf{Det}}{\slashed{D}})$.  As a generalized cohomology class, this class will measure the so-called \emph{local and global anomaly}, revealing crucial information about the family of operators $\slashed{D}_\tau$.  Its computation is achieved by the Riemann-Roch-Grothendieck-Quillen (RRGQ) formula and relies on the functional determinant of the elliptic differential operator $\slashed{D}^\dagger_\tau\slashed{D}_\tau$.
\par This article is structured as follows: in Section~\ref{sec:signature}, we introduce the fundamental notion of the signature of a manifold both as a topological and an analytical index. In Section~\ref{cobordism}, we  investigate the structure of the oriented cobordism ring and the rational homomorphisms on it. In the context of the signature of a manifold, this gives rise to the Hirzebruch $L$-polynomials. In Section~\ref{sec:det} we construct Quillen's determinant line bundle for the family of (complexified) signature operators over a Jacobian elliptic surface and compute its first Chern class by the RRGQ formula. In Section~\ref{section:gravanom} we combine several anomalous operators by coupling the original family of operators to a holomorphic $\operatorname{SU}(2)$ bundle over the elliptic surface. Using the Poincar\'e line bundle over the elliptic surface, we prove that for certain Jacobian elliptic surfaces, a holomorphic $\operatorname{SU}(2)$ bundle can always be constructed in such a way that the local anomaly vanishes.
\subsection{The Cobordism Ring and Genera}
One approach to investigating manifolds is through the construction of
invariants. In this article, we work to construct topological and geometric invariants in the category $\mathcal{S}$ of smooth, compact, orientable Riemannian manifolds. Given two such manifolds, we may generate new objects through the operations ${\bf +}$ (disjoint union), ${\bf -}$ (reversing the orientation), and ${\bf \times}$ (Cartesian product). These operations turn $\mathcal{S}$ into a graded commutative monoid, graded naturally by dimension, which decomposes as the direct sum \beq \mathcal{S} = \bigoplus_{n=0}^{\infty} \mathcal{M}_n \,. \eeq Here $\mathcal{M}_n \subseteq \mathcal{S}$ is the class of all smooth, compact, oriented Riemannian manifolds of dimension $n$. In $\mathcal{S}$, the additive operation is only well defined when restricted to elements in the same class $\mathcal{M}_n$. Furthermore, if $M \in \mathcal{M}_k$ and $N \in \mathcal{M}_j$, we have the graded commutativity relation $M \times N = (-1)^{k j}N \times M$. The positively oriented point $\{*\}$ serves as the unity element $1_{\mathcal{S}}$, but $\mathcal{S}$ fails to be a (graded) semi-ring because there is no neutral element $0_{\mathcal{S}}$ and there are no additive inverses. It is reasonable to attempt to define $0_{\mathcal{S}}=\emptyset$ and the additive inverse of $M$ as $-M$ by reversing the orientation. However, $M \cup - M \neq \emptyset$, so the additive inverse is not well defined. If we insist that we want such an assignment of additive inverses and $0_{\mathcal{S}}$ to turn $\mathcal{S}$ or some quotient into a ring, we are led naturally to the notion of \emph{cobordant} manifolds. 
\begin{definition}
An oriented differentiable manifold $V^n$ bounds if
there exists a compact oriented manifold $X^{n+1}$ with oriented
boundary $\partial X^{n+1}=  V^n$.
Two manifolds $V^n, W^n \in \mathcal{M}_n$ are cobordant if $V^n - W^n = \partial X^{n+1}$
for some smooth, compact $(n+1)$-manifold $X$ with boundary, where by $V^n - W^n$ we mean $V^n \cup (-W^n)$ by reversing the orientation on $W^n$.  
\end{definition}
The notion of ``cobordant'' is an equivalence relation $\sim$ on the class $\mathcal{M}_n$; hence we introduce the \emph{oriented cobordism ring} $\Omega_*$ as the collection of all equivalence classes $[M] \in \Omega_n = \mathcal{M}_n / \sim$ for all $n \in \mathbb{N}$,

$$\Omega_* = \bigoplus_{n=0}^{\infty} \Omega_n.$$

\vspace{3mm}

Note first that for any boundary $V^n=\del X^{n+1}$, $V^n$ descends to the zero element $0=[\emptyset]\in \Omega_*$. Since for any manifold $M$, we have that $M - M = \del(M \times [0,1])$, the boundary of the oriented cylinder bounded by $M$, it follows that on $\Omega_*$, $[M-M]=0$. Thus $-[M]=[-M]$ for any manifold $M$. This in turn shows that the oriented cobordism ring $\Omega_*$ naturally inherits the structure of a graded commutative ring with respect to the operations on $\mathcal{S}$. For a thorough survey of cobordism (oriented and otherwise) we refer the reader to \cite{MR0248858}. 

To study these structures, we look at nontrivial ring homomorphisms $\psi : \Omega_* \to \mathbb{C}$, i.e., for any two suitable $V,W \in \Omega_*$, the morphism $\psi$ satisfies 
\beq
 \begin{array}{lclrcl}
 \psi( V + W) & = & \psi(V) + \psi(W) \,, &
 \psi( -V) & = & - \psi(V) \,, \\
 \psi( V \times W) &=& \psi(V) \cdot \psi(W)\,.
 \end{array}
\eeq
Thus, we are interested in studying the dual space of $\Omega_*$. This prompts the following:
\begin{definition}
 A genus is a ring homomorphism $\psi : \Omega_* \to \mathbb{Q}$. Hence, a genus $\psi \in \operatorname{hom}_{\mathbb{Q}}(\Omega_*;\mathbb{Q})$ is an element of the rational dual space to $\Omega_*$.
\end{definition}
This implies that if $V^n=\partial X^{n+1}$ bounds, then necessarily $\psi(V^n)=0$, as $\psi$ must be compatible with the ring operations $(+,-,\times)$ and constant on the equivalence classes of manifolds that satisfy $V^n-W^n=\partial X^{n+1}$. 
\section{The signature of a manifold}
\label{sec:signature}
In this section, we introduce the fundamental notion of the signature of a manifold of dimension $4k$. This quantity constitutes one of the most prominent examples of  a ring homomorphism $\Omega_* \to \mathbb{Q}$. In fact, if the dimension of a compact manifold is divisible by four, the middle cohomology group is equipped with a real valued symmetric bilinear form $\iota$, called the \emph{intersection form}. The properties of this bilinear form allow us to define an important topological invariant, the \emph{signature of a manifold}. The \emph{Hirzebruch signature theorem} asserts that this topological index equals  the analytic index of an elliptic operator on $M$, if $M$ is equipped with a smooth Riemannian structure. 
\subsection{The definition of the signature} 
Let $M$ be a compact manifold of dimension $n=4k$. On the middle cohomology group $H^{2k}(M;\mathbb{R})$, a symmetric bilinear form $\iota$ called the \emph{intersection form}, is obtained by evaluating the cup product of two $2k$-cocycles $a, b$ on the fundamental homology class $[M]$ of $M$. Alternatively, we can obtain $\iota$ by integrating the wedge product of the corresponding differential $2k$-forms $\omega_a,\omega_b$ over $M$ via the de\!~Rham isomorphism. That is, we define $\iota$ and $\iota_{dR}$ as follows:

\beq \begin{array}{lcrclcc}
\iota: && H^{2k}(M;\mathbb{R}) & \otimes & H^{2k}(M;\mathbb{R}) & \to & \mathbb{R}\\ 
&& & & (a, b) &\mapsto & \langle a \cup b, \, \lbrack M \rbrack\rangle \\ 
\iota_{dR}: && H^{2k}_{dR}(M;\mathbb{R}) & \otimes &H^{2k}_{dR}(M;\mathbb{R}) & \to & \mathbb{R}
\\ && & & (\omega_a, \omega_b) &\mapsto& \int_{M} \omega_a \wedge \omega_b \;.
\end{array}
\eeq 
Notice that $\iota$ is in fact symmetric since the dimension is a multiple of four.
\begin{remark}
Since $M$ is compact, Poincar\'e duality asserts that for any cocycle $a \in H^{2k}(M;\mathbb{R})$ there is a corresponding cycle $\alpha \in
H_{2k}(M;\mathbb{R})$. If we assume that the two cycles intersect transversally, then $\iota(a, b)$ can be expressed 
as the intersection number of the two cycles $\alpha$ and $\beta$, i.e., by computing
\beq
 \iota(a, b) = \#(\alpha,\beta) = \sum_{p \in \alpha \cap \beta} \,
i_p(\alpha,\beta)
\eeq
where $i_p(\alpha,\beta)$ is either $+1$ or $-1$ depending on whether
the orientation of the $T_pM$ induced by the two cycles $\alpha,
\beta$ agrees with the orientation of the manifold or not \cite{MR0061823}.
\end{remark}
Sylvester's Theorem guarantees that any non-degenerate, real valued, symmetric bilinear form on a finite dimensional vector space can
be diagonalized with only $+1$ or $-1$ entries on the diagonal.  Thus, after a change of basis, we have $\iota \cong I_{u, v}$, 
where $I_{u, v}$ is the diagonal matrix with $u$ entries $+1$ and $v$ entries $-1$.
This prompts the following definition.
\begin{definition}
Let $M$ be a compact $4k$-dimensional manifold with intersection form $\iota$ on the middle cohomology group $H^{2k}(M;\mathbb{R})$ such that $\iota \cong I_{u, v}$. Then the signature of $M$ is given by $\mathrm{sign}(M)=u-v$. If the dimension of $M$ is not a multiple of four, the signature is defined to be zero.
\end{definition}
\par It turns out that the signature is a fundamental topological \emph{invariant} of $M$: if two $4k$ manifolds are homeomorphic, their signature will be equal; in fact, the signature is an invariant of the oriented homotopy class of $M$ \cites{MR0172306, MR0082103}. Furthermore,  the signature is constant on the oriented cobordism classes in $\Omega_*$. We have the following theorem due to Hirzebruch \cite{MR0063670}. 

\begin{theorem}
\label{thm:sign}
The signature defines a ring homomorphism $\mathrm{sign} : \Omega_* \to \mathbb{Z}$. In particular, for all $k \in \mathbb{N}$ we have $\mathrm{sign}(\mathbb{C} P^{2k})=1. $\end{theorem}
The theorem is proved by checking that the signature is compatible with the ring operations $(+,-,\times)$ on the classes $\mathcal{M}_{4k}\subseteq \mathcal{S}$, $k=1,2,\dots,$ by using the K\"unneth theorem and writing out the definition of the signature on a basis for the middle cohomology. Furthermore, one can show that $\mathrm{sign}(M)=0$ for any manifold $M=\del X$ that bounds. This is done by constructing a commutative exact ladder for $i: V^{4k} \hookrightarrow X^{4k+1}$, i.e., the embedding of $V^{4k}$ as the boundary of $X^{4k+1}$ using Lefschetz and Poincar\'e duality.
The normalization for the even dimensional complex projective spaces $\mathbb{C} P^{2k}$ can be checked as follows: the cohomology ring
$H^*(\mathbb{C} P^{2k};\mathbb{R})$ is a truncated polynomial ring in dimension $4k$, generated by the first Chern class of the hyperplane bundle $\mathrm{H} \in H^2(\mathbb{C} P^{2k};\mathbb{R})$. It follows that $H^{2k}(\mathbb{C} P^{2k};\mathbb{R})$ is generated by $\mathrm{H}^k$, and 
\beqn
\label{eqn:normalization}
\iota(\mathrm{H}^k,\mathrm{H}^k)= \mathrm{H}^{2k} \lbrack \mathbb{C} P^{2k} \rbrack = 1. 
\eeqn
\subsection{The signature as an analytical index}
\label{ssec:index}
We will identify the signature of an $n$-dimensional smooth, compact oriented manifold $M$ with the \emph{index} of an elliptic operator acting on sections of a smooth vector bundle over $M$.\begin{footnote}{The results of this section can also be phrased in terms of spinors and spin bundles. We have chosen to leave this viewpoint out as to make the content more accessible.}\end{footnote} 
\par The operator in question is the \emph{Laplace} operator acting on smooth, complex valued differential $p$-forms on $M$. Over any oriented closed $n$-manifold $M$ we have the exterior product bundles $\Lambda^p=\Lambda^p T^{*}_\mathbb{C} M$ of the complexified cotangent bundle $T^*_\mathbb{C} M = T^*M \otimes \mathbb{C}$. The complex-valued $p$-forms are the smooth sections $\omega : M \to \Lambda^p$; they form a vector space which we denote by $C^\infty(M,\Lambda^p)$. The vector spaces can be assembled into the so-called \emph{de\!~Rham complex} $C^{\infty}(M, \Lambda^*)=\bigoplus_{p=0}^{n}C^{\infty}(M, \Lambda^p)$. The summands are connected by the exterior derivative $d: C^\infty(M, \Lambda^p) \to C^\infty(M, \Lambda^{p+1})$, extended linearly over $T^*_{\mathbb{C}}M$, with the usual relation $d^2=0$. 
\par Moreover, a Riemannian metric on $M$ determines a Hermitian structure $(\cdot\,,\cdot)$ on $C^\infty(M, \Lambda^p)$ via the \emph{Hodge-de Rham operator} $*:
C^\infty(M, \Lambda^p) \to C^\infty(M, \Lambda^{n-p})$. The Hodge-de Rham operator is a bundle isomorphism $C^\infty(M, \Lambda^p) \to C^\infty(M, \Lambda^{n-p})$ which satisfies $*^2=(-1)^{p(n-p)}\mathbb{I}$, where $\mathbb{I} : C^{\infty}(M, \Lambda^p) \to  C^{\infty}(M, \Lambda^p)$ is the identity. Specifically, the Hodge dual of a $p$-form $\omega \in C^{\infty}(M, \Lambda^p)$ is the $(n-p)$-form denoted by $*\omega \in C^{\infty}(M, \Lambda^{n-p})$ determined by the property that for any $p$-form $\mu \in C^{\infty}(M, \Lambda^p)$ we have
 \beq \mu \wedge *\omega = \langle \mu , \omega \rangle \mathrm{vol}_M \,,\eeq where the Riemannian structure on $M$ induces both the inner product $\langle \cdot , \cdot \rangle$ on $C^{\infty}(M, \Lambda^p)$ and the volume form  $\mathrm{vol}_M$. A Hermitian structure on $C^{\infty}(M, \Lambda^p)$ is then defined by setting 
   \beq
(\mu, \omega) = \int_{M} \mu \wedge *\overbar{\omega}\,,
\eeq 
for any two $p$-forms $\mu,\omega$ where $\overbar{\omega}$ means complex conjugation. With respect to this inner product, we obtain an adjoint operator
\beq
 \delta: C^\infty(M, \Lambda^{p+1}) \to C^\infty(M, \Lambda^p)
 \eeq 
of the exterior derivative $d$ defined by $(d\eta^p,\omega^{p+1}) = (\eta^p, \delta \omega^{p+1})$ and $\delta^2=0$. If $n=2m$ the adjoint operator is $\delta = - *d*$. 
\par The operator $D=d+\delta : C^\infty(M, \Lambda^*) \to C^{\infty}(M, \Lambda^*)$ is a first order differential operator that, by construction, is formally self-adjoint. We call $D$ a \emph{Dirac} operator, because it is, in a sense, the square root of the Laplace operator. In fact, the \emph{Laplace operator}, also called the \emph{Hodge Laplacian}, is given as its square by 
\beq
\Delta_H =  (d + \delta)^2 = d\delta + \delta d \,.
\eeq 
 The Hodge Laplacian is homogeneous of degree zero, i.e., $\Delta_H: C^\infty(M,\Lambda^p) \to C^\infty(M,\Lambda^p)$, and is formally self-adjoint, i.e., $(\Delta_H \omega, \mu) = (\omega, \Delta_H \mu)$ for any $p$-forms $\omega, \mu$.  A $p$-form $\omega$ is said to be \emph{harmonic} if $\Delta_H \omega = 0$. One then shows that $\omega$ is harmonic if and only if $\omega$ is both closed and co-closed, i.e., 
\beqn   
\label{eqn:harmonic_rep}
\Delta_H \omega = 0 \; \Leftrightarrow \; d\omega = 0\,, \delta \omega =0 \,. 
\eeqn 
In fact, every $\eta \in H_{dR}^p(M;\mathbb{C})$ has a unique harmonic representative $\omega$ such that $\eta = \omega + d\phi$ for some $(p-1)$-form $\phi$; this is the celebrated Hodge theorem \cite{MR0003947}. This implies that there is an isomorphism $H_{dR}^*(M,\mathbb{C}) \cong \ker\Delta_H$. 
\par Let us restrict to the case of an even-dimensional manifold $M$, i.e., $n=2m$. Then, for $p=0,\dots,2m$ we define the complex operators 
\beq 
 \alpha_p=\sqrt{-1}^{p(p-1)+m}*:C^\infty(M,\Lambda^p) \to C^\infty(M,\Lambda^{2m-p}),
 \eeq 
 which for complex-valued differential forms are more natural than the Hodge-de\!~Rham operator. We will denote the operators generically by $\alpha :  C^\infty(M,\Lambda^*) \to C^{\infty}(M,\Lambda^*)$.  One may check that the operators $\alpha_{2m-p}$ and $\alpha_p$ are inverses, $\alpha_{2m-p}\alpha_p=\alpha_p\alpha_{2m-p}=\mathbb{I}$. Hence, the eigenvalues of $\alpha_m$ are $\pm 1$, so the de\!~Rham complex splits into the two eigenspaces of $\alpha$, given by $\alpha(\omega)=\pm \omega$, such that
\beq 
C^\infty(M,\Lambda^*) & = &
\underbrace{C_+^\infty(M,\Lambda^*)}_{\mbox{\tiny eigenvalue: $+1$}} \
\bigoplus \ \underbrace{C_-^\infty(M,\Lambda^*)}_{\mbox{\tiny eigenvalue:$-1$}} \,.
\eeq 
The projection operators onto the two eigenspaces given by $P_+ = (\mathbb{I}+ \alpha)/2$ and  $P_- = (\mathbb{I}- \alpha)/2$, respectively, and $P_+ + P_- =\mathbb{I}$, $P_\pm P_\mp =0$. Since $\alpha$ anti-commutes with $d+\delta$, i.e., $\alpha (d+\delta) = -(d+\delta) \alpha$ and because of Equation~(\ref{eqn:harmonic_rep}), we get an orthogonal, direct-sum decomposition of the entire de\!~Rham cohomology according to
\beqn
\label{eqn:decomp}
 H_{dR}^*(M;\mathbb{C}) & = & \bigoplus_{p=0}^m \ H^p_+(M;\mathbb{C}) \oplus H^p_-(M;\mathbb{C})\,, 
 \eeqn 
where $H^p_\pm(M;\mathbb{C})= \lbrace \omega \in H_{dR}^p(M;\mathbb{C}) \oplus H_{dR}^{2m-p}(M;\mathbb{C}) \mid \alpha(
\omega) = \pm \omega \rbrace$. Notice that  for $p \not = m$, every element of $H^p_\pm(M;\mathbb{C})$ necessarily
has the form $\omega \pm \alpha(\omega)$ for a non-trivial element $\omega \in H_{dR}^p(M;\mathbb{C})$, and we have for all $p \not =m$ the identity
\beqn
\label{eqn:cancellation}
 \dim H^p_+(M,\mathbb{C}) = \dim H^p_-(M,\mathbb{C}) \,.
\eeqn 
\par We now recall the notion of an \emph{elliptic} differential operator, and the notion of its analytic index \cites{MR0161012, MR0198494}. Suppose we have two vector bundles $E,F \to M$ and a $q^{th}$-order differential operator $\slashed{D} : C^{\infty}(M,E) \to C^{\infty}(M,F)$. Let $T'M$ be the cotangent bundle of $M$ minus the zero section and $\pi: T'M \to M$ be the canonical projection. Then the \emph{principal symbol} of the operator $\slashed{D}$ is a linear mapping $\sigma_{\slashed{D}} \in \operatorname{Hom}(\pi^*E,\pi^*F)$ such that $\sigma_{\slashed{D}}(x,\rho \xi) = \rho^q\sigma(x,\xi)$ for all $(x,\xi) \in T'M$ and for all scalars $\rho$. The operator $\slashed{D}$ is \emph{elliptic} if the principal symbol $\sigma_{\slashed{D}}$ is a fiberwise isomorphism for all $x \in M$.  It is a classical result that for an elliptic operator, the kernel and cokernel are finite dimensional \cite{MR1554993}. The \emph{analytic index} of $\slashed{D}$ is then defined as \beq \operatorname{ind}(\slashed{D})=\dim(\ker{\slashed{D}}) - \dim(\operatorname{coker}\slashed{D}).\eeq
\par If we restrict the operator $D=d + \delta$ to the $\pm1$-eigenspaces of $\alpha$, these restrictions become formal adjoint operators of one another. In physics, they are called \emph{chiral} Dirac operators. We denote these operators by
\beqn
\label{eqn:chiral}
\slashed{D}  := (d+\delta) P_+ = P_- (d+\delta)  : \; C_+^\infty(M,\Lambda^*) & \to & C_-^\infty(M,\Lambda^*) \,,\\ 
\slashed{D}^\dagger  := (d+\delta) P_- = P_+ (d+\delta)  : \; C_-^\infty(M,\Lambda^*) & \to & C_+^\infty(M,\Lambda^*) \,,
\eeqn
where we again used the fact that $\alpha$ anti-commutes with $D=d+\delta$. We have that 
\beq 
 \slashed{D}^{\dagger}\slashed{D}=\Delta_H P_+ = \Delta_H \Big{|}_{C^{\infty}_+(M,\Lambda^*)},
\eeq
and that the operators $\slashed{D}$ and $\slashed{D}^\dagger$ are elliptic operators such that $\ker{\slashed{D}^\dagger}=\operatorname{coker}{\slashed{D}}$. To see this, we look to the principle symbols of the first order differential operators $d$ and $\delta$ on the vector bundle $C^{\infty}(M,\Lambda^*)$ over $M$. One computes their principal symbols as $\sigma_d(x,\xi)=\operatorname{ext}(\xi)$ and $\sigma_{\delta}(x,\xi)=-\operatorname{int}(\xi)$, respectively, signifying the exterior algebra homomorphisms induced from the exterior product and interior product of forms, respectively. This implies that $\slashed{D}+\slashed{D}^\dagger = D$ is the Clifford multiplication on the exterior algebra of forms. Therefore, $\sigma_{\Delta_H}(x,\xi)=-|\xi|^2 \mathbb{I}$. This mapping is invertible, and it follows that $\slashed{D}$, $\slashed{D}^\dagger$, and $\Delta_H$ are elliptic operators. 
\par For the index of the elliptic operator $\slashed{D}$ on a compact even-dimensional manifold $M^n$ with $n=2m$ we obtain
 \beq 
 \operatorname{ind}(\slashed{D}) & = & \dim( \ker \slashed{D}) - \dim (\ker \slashed{D}^\dagger)\\
& = & \sum_{p=0}^m\dim H^p_+(M;\mathbb{C}) - \dim H^p_-(M;\mathbb{C}) \\ 
& = & \dim H^m_+(M;\mathbb{C}) - \dim H^m_-(M;\mathbb{C}) \,,  \eeq 
where we used $\ker{\slashed{D}^\dagger}=\operatorname{coker}{\slashed{D}}$ and Equation~(\ref{eqn:cancellation}).
\par For $n=2m$ with $m=2k+1$ the operator $\alpha_m$ is an isomorphism between $H^m_+(M;\mathbb{C})$ and
$ H^m_-(M;\mathbb{C})$, thus $\operatorname{ind}(\slashed{D})=0$. In contrast, for $n=4k$ we have 
\beq
  \operatorname{ind}(\slashed{D}) & = & \dim{H^{2k}_+(M;\mathbb{C})} -\dim{H^{2k}_-(M;\mathbb{C})} \,,
\eeq  
and $\alpha$ maps $H^{2k}_\pm(M;\mathbb{C})$ to itself. Moreover, for $n=4k$ we have $\alpha=*$ and the decomposition into $\pm1$-eigenspaces of $\alpha$ coincides with the orthogonal, direct-sum decomposition of $H_{dR}^{2k}(M;\mathbb{C})$ into self-dual and anti-self-dual, middle-dimensional, complex differential forms. Thus, we have
\beq 
H_{dR}^{2k}(M;\mathbb{C}) & = & H^{2k}_+(M;\mathbb{C}) \oplus H^{2k}_-(M;\mathbb{C}) \,.
\eeq 
On the other  hand, using the de\!~Rham isomorphism, we can evaluate the intersection form $\iota_{dR}$ on either two self-dual or anti-self-dual $2k$-forms $\omega, \mu$, to obtain
\beq 
\iota_{dR}(\omega,\mu) & = &(\omega,\mu) \,,\\
\Rightarrow \qquad \iota_{dR}(\omega,\omega)\big|_{\omega \in H^{2k}_\pm(M,\mathbb{C})} & = & (\omega,\omega) = \pm \Vert \omega\Vert ^2 \,. 
\eeq
This shows that $\iota_{dR}\cong I_{u, v}$ with $u =\dim{H^{2k}_+(M;\mathbb{C})}$ and $v =\dim{H^{2k}_-(M;\mathbb{C})}$,
so that the index of the Dirac operator is equal to the difference of the number of linearly independent self-dual and anti-self-dual cohomology classes over $\mathbb{C}$. Hence, we have the following special case of the Atiyah-Singer index theorem \cites{MR157392, MR0063670}:
\begin{theorem}[Analytic index of signature operator]
\label{thm:analytic_sign}
For all compact, oriented Riemannian manifolds $M$ of dimension $4k$, the (analytic) index of the Dirac operator defined above equals the signature of $M$, i.e.,  $\operatorname{ind}({\slashed{D}})= \operatorname{sign}(M)$.
\end{theorem}
\section{The structure of the Cobordism ring and its genera}
\label{cobordism}
Having shown that the signature of a manifold $M$ is a genus, we now investigate the structure of the oriented cobordism ring and its rational homomorphisms. In this spirit, it is beneficial to look at the \emph{rational} oriented cobordism ring $\Omega_* \otimes \mathbb{Q}$. Tensoring with $\mathbb{Q}$ kills torsion subgroups; in fact, it is known that the cobordism groups $\Omega_n$ are finite and (graded) commutative when $n\neq 4k$ \cite{MR0061823}, implying that each such cobordism group is torsion. We shall see that $\Omega_* \otimes \mathbb{Q}$ has a generator of degree $4k$ for all $k\in \mathbb{N}$. This implies that the graded commutativity of $\Omega_*$ is erased by tensoring with $\mathbb{Q}$, leaving an honest commutative ring. Thus, the rational cobordism ring is of the form \beq \Omega_* \otimes \mathbb{Q} = \bigoplus_{k=0}^{\infty} \Omega_{4k} \otimes \mathbb{Q} \eeq and forms a commutative ring with unity. We will soon be able to say much more about the structure of this ring and provide a generating set, along with a complete description of its dual $\hom_\mathbb{Q}\big(\Omega_* \otimes \mathbb{Q}, \mathbb{Q}\big)$. This will allow for a description of the signature of a $4k$-manifold as a topological index in terms of its Pontrjagin classes.
\subsection{Pontrjagin numbers}
We begin with the following motivation for the Pontrjagin classes \cite{MR0440554}. Let $\xi$ be a real vector bundle of rank $r$ over a topological space $B$. Then the \emph{complexification} $\xi \otimes \mathbb{C} = \xi \otimes_{\mathbb{R}} \mathbb{C}$ is a complex rank $r$ vector bundle over $B$, with a typical fiber $F \otimes \mathbb{C}$ given by \beq F \otimes \mathbb{C} = F \oplus \sqrt{-1}F. \eeq It follows that the underlying real vector bundle $(\xi \otimes \mathbb{C})_{\mathbb{R}}$ is canonically isomorphic to the Whitney sum $\xi \oplus \xi$. Furthermore, the complex structure on $\xi \otimes \mathbb{C}$ corresponds to the bundle endomorphism $J(x, y)=(-y, x)$ on $\xi \oplus \xi$. 
\par In general, complex vector bundles $\eta$ are not isomorphic to their conjugate bundles $\overbar{\eta}$. However, complexifications of real vector bundles are \emph{always} isomorphic to their conjugates: we have $\xi \otimes \mathbb{C} \cong \overbar{\xi \otimes \mathbb{C}}$.  Let us consider the total Chern class given as the formal sum
\beq c(\xi \otimes \mathbb{C}) = 1 + c_1(\xi \otimes \mathbb{C}) + c_2(\xi \otimes \mathbb{C}) + \dots + c_r(\xi \otimes \mathbb{C}).
\eeq 
We refer the reader to \cite{MR0015793} for an introduction to Chern classes. From the above isomorphism, this expression is equal to the total Chern class of the conjugate bundle, that is
\beq 
c(\overbar{\xi \otimes \mathbb{C}}) = 1 - c_1(\overbar{\xi \otimes \mathbb{C}}) + c_2(\overbar{\xi \otimes \mathbb{C}}) - \dots + (-1)^r c_r(\overbar{\xi \otimes \mathbb{C}}).
\eeq 
Thus, the odd Chern classes $c_{2i+1}(\xi \otimes \mathbb{C})$ are all elements of order 2, so it is the even Chern classes of the complexification $\xi \otimes \mathbb{C}$ that encode relevant topological information about the real vector bundle $\xi$. This prompts the following definition.

\begin{definition}
Let $\xi \to B$ be a real vector bundle of rank $r$. The $i^{th}$ rational Pontrjagin class of $\xi$ for $1 \le i \le r$ is defined as
 \beq p_i(\xi) & = &  (-1)^i \, c_{2i}(\xi \otimes \mathbb{C}) \in H^{4i}(B;\mathbb{Q}) \,. \eeq 
where $c_{2i}$ is the $2i^{th}$ rational Chern class of the complexified bundle $\xi \otimes \mathbb{C}$.The total rational Pontrjagin class is the formal sum 
\beq p(\xi) &=& \sum_i p_i(\xi) \,\in H^*(B;\mathbb{Q}).\eeq The Pontrjagin classes of an oriented
manifold $M$ are the Pontrjagin classes of its tangent bundle $TM$. 
\end{definition}

\par As a practical way to compute the Pontrjagin classes of $\xi$, let $\nabla$ be a bundle connection on $\xi \otimes \mathbb{C}$, and let $\Omega$ be the corresponding curvature tensor. Then the total Pontrjagin class $p(\xi) \in H^*_{dR}(B;\mathbb{R})$ is a polynomial in the curvature tensor given by the expansion of the right hand side in the formula
\beq 
 p(\xi) = 1 + p_1(\xi) + p_2(\xi) + \dots + p_r(\xi) = \det(I+\Omega/2\pi) \in  H^*_{dR}(B;\mathbb{R}). 
 \eeq 
As the total Pontrjagin class is defined as a determinant, it is immediately invariant under the adjoint action of $\operatorname{GL}(r,\mathbb{R})$, and furthermore, that the cohomology classes of the resulting differential forms are independent of the connection $\nabla$ \cites{MR0045432, MR0011027}. 

Suppose that $M$ is an oriented $4k$-dimensional manifold. If $(i)=(i_1, \dots, i_m)$ is a partition of $k$, i.e. $|(i)|=\sum_{a=1}^m
i_a =k$, then we define the corresponding \emph{rational Pontrjagin number} of $M$ to be
\beqn 
\label{eqn1}
p_{(i)}\lbrack M\rbrack = \big( p_{i_1}(TM) \cup \cdots \cup p_{i_m}(TM) \big) \lbrack M \rbrack \in \mathbb{Q} \,,
\eeqn
and, if the dimension is not a multiple of four, all Pontrjagin numbers of $M$ are defined to be zero. When working with the Pontrjagin classes in the de\!~Rham cohomology of $M$, the Pontrjagin numbers are given by integrating the top-dimensional form $p_{(i)} = p_{i_1} \wedge  p_{i_2} \wedge \cdots \wedge  p_{i_m}$ over the fundamental cycle $[M]$, i.e.,  
\beqn
\label{eqn2}
p_{(i)}[M] = \int_M  p_{i_1} \wedge  p_{i_2} \wedge \cdots \wedge  p_{i_m}. 
\eeqn
\begin{example}
\label{example:signature}
Let us compute the Pontrjagin number for the class $p_1/3$ on $\mathbb{C} P^2$. Equip the tangent bundle of $\mathbb{C} P^2$ with the Fubini-Study metric \cites{Fubini, MR1511309}. Then in the affine coordinate chart $[z_1:z_2:1]$, a computation with the curvature tensor shows that 
\beq 
 \frac{1}{3}p_1 = \frac{2}{\pi^2(|z_1|^2+|z_2|^2+1)^3}dz_1 \wedge d\overbar{z}_1 \wedge dz_2 \wedge d\overbar{z}_2.
\eeq 
Integrating this form over the fundamental cycle yields  
\beq 
   \int_{\mathbb{C} P^2} \frac{1}{3}p_1 =1.
\eeq 
As we shall see in Section \ref{section:hirzebruchsignature}, this is consistent with $\operatorname{sign}(\mathbb{C} P^2)=1$ by Theorem \ref{thm:sign}.
\end{example}
\par The following explains the importance of the Pontrjagin classes \cite{MR0022667}:
\begin{proposition}[Theorem of Pontrjagin]
\label{rem1}
The evaluation of the Pontrjagin classes are compatible with the operations $(+,-,\times)$ on $\Omega_* \otimes \mathbb{Q}$. Furthermore, all Pontrjagin numbers vanish if a manifold bounds. Thus, the Pontrjagin numbers define rational homomorphisms $\Omega_* \otimes \mathbb{Q} \to \mathbb{Q}$.
\end{proposition}
In Proposition~\ref{rem1} compatibility with $+$ is trivial and compatibility with $-$ follows from the fact that the Pontrjagin numbers are independent of the orientation \cites{MR0180989, MR0193644}. For multiplicativity, notice that the total Pontrjagin class of the Whitney sum satisfies $p(\xi_1 \oplus \xi_2) = p(\xi_1)p(\xi_2)$ up to two-torsion \cite{MR0022667}. As we are working over $\mathbb{Q}$, this equality always holds. 

\subsection{The rational cohomology ring and genera}
We have seen in Proposition~\ref{rem1} that the Pontrjagin numbers constitute prototypes of the rational homomorphisms $\Omega_*\otimes \mathbb{Q} \to \mathbb{Q}$. It turns out that every genus can be constructed in this way \cite{MR0061823}. This follows from a remarkable relationship between the rational oriented cobordism ring $\Omega_* \otimes \mathbb{Q}$ and the cohomology ring of the classifying space $\operatorname{BSO}$ for oriented vector bundles.
\par  We assume that $M$ is a smooth, compact $n$-manifold and recall the notion of a \emph{classifying space} for oriented vector bundles.

\begin{definition}
For every oriented rank $r$ vector bundle $\pi: E \to M$, there is a topological space $\operatorname{BSO}(r)$ and a vector bundle 
$\operatorname{ESO}(r) \to \operatorname{BSO}(r)$ together with a classifying map $f : M \to \operatorname{BSO}(r)$ such that $f^*\operatorname{ESO}(r)=E$. The space $\operatorname{BSO}(r)$ is called a classifying space, and the vector bundle $\operatorname{ESO}(r) \to \operatorname{BSO}(r)$ is called the universal bundle of oriented $r$-planes over $\operatorname{BSO}(r)$.
\end{definition}
It is known that the isomorphism class of the vector bundle $E$ determines the classifying map $f$ up to homotopy. Moreover, the classifying space and the universal bundle can be constructed explicitly as the Grassmannian of oriented $r$-planes over $\mathbb{R}^{\infty}$, denoted by $\operatorname{BSO}(r)=\operatorname{Gr}_r(\mathbb{R}^{\infty})$, and the tautological bundle of oriented $r$-planes over it, denoted by $\operatorname{ESO}(r) = \gamma^r \to \operatorname{BSO}(r)$ respectively \cites{MR0077122,MR0077932}.
\par For any oriented rank $r$ vector bundle $E \to M$ we can form the sphere bundle $\Sigma(E) \to M$ by taking the one-point compactification of each fiber $E_p$ of the vector bundle over each point $p \in M$ and gluing them together to get the total space. We then construct the \emph{Thom space} $T(E)$ \cite{MR0054960} as the quotient by identifying all the new points to a single point $t_0=\infty$, which we take as the base point of $T(E)$. Since $M$ is assumed compact, $T(E)$ is the one-point compactification of $E$. Associated to the universal bundle $\operatorname{ESO}(r)$ is the $\emph{universal Thom space}$, denoted by $\operatorname{MSO}(r)$. 
\par Moreover, thinking of $M$ as embedded into $E$ as the zero section, there is a class $u \in \tilde{H}^r( T(E); \mathbb{Z})$, called the \emph{Thom class}, such that for any fiber $E_p$ the restriction of $u$ is the (orientation) class induced by the given orientation of the fiber $E_p$. This class $u$ is naturally an element of the reduced cohomology \cite{MR0054960}
\beq
 \tilde{H}^r( T(E); \mathbb{Z}) \cong H^r( \Sigma(E), M ; \mathbb{Z}) \cong H^r( E, E- M ; \mathbb{Z}) \,.
\eeq
It turns out that the map
\beq
 H^i\big( E ; \mathbb{Z} \big) \to \tilde{H}^{i+r}\big( T(E) ; \mathbb{Z} \big) \,, \qquad z \mapsto z \cup u \,,
\eeq 
is an isomorphism for every $i \ge 0$, called the \emph{Thom isomorphism} \cite{MR0054960}.  Since the pullback map $\pi^*: H^*(M;\mathbb{Z}) \to H^*(E;\mathbb{Z})$ is a ring isomorphism as well, we obtain an isomorphism
\beqn
\label{eqn:ThomIso}
 H^i\big( M ; \mathbb{Z} \big) \to \tilde{H}^{i+r}\big( T(E) ; \mathbb{Z} \big) \,, \qquad x \mapsto \pi^*(x) \cup u \,,
\eeqn 
for every $i \ge 0$ which sends the identity element of $H^*(M;\mathbb{Z})$ to $u$. The following result of Thom is crucial \cite{MR0061823}:
\begin{theorem} (Pontrjagin-Thom)\\
For $r \geq n+2$, the homotopy group $\pi_{n+r}(\operatorname{MSO}(r),\infty)$ is isomorphic to the oriented cobordism group $\Omega_n$. 
\end{theorem}
\begin{proof} We offer only a sketch, and refer the reader to \cites{MR0054960, MR0061823, MR0440554} for more details. Let $M$ be an arbitrary smooth, compact, orientable $n$-manifold, and let $E \to M$ be a smooth, oriented vector bundle of rank $r$. The base manifold $M$ is smoothly embedded in the total space $E$ as the zero section, and hence in the Thom space $T=T(E)$. In particular, note that $M \subset T - \{ t_0\}$. Note that $T$ itself is not a smooth manifold -- it is singular precisely at the base point $t_0$. Results in \cite{MR0440554} show that every continuous map $f : S^{n+r} \to T$ is homotopic to a map $\widehat{f}$ that is smooth on $\widehat{f}^{-1}(T - \{\ t_0 \})$. It follows that the oriented cobordism class of $\widehat{f}^{-1}(M)$ depends only on the homotopy class of $\widehat{f}$. Hence, the mapping $\widehat{f} \mapsto \widehat{f}^{-1}(M)$ induces a homomorphism $\pi_{n+r}(T,t_0) \to \Omega_n$. We will argue that this homomorphism is surjective. We refer the reader to \cite{MR0061823} for the proof of injectivity.

A theorem of Whitney \cite{MR0010274} shows that $M$ can be smoothly embedded in $\mathbb{R}^{n+r}$. Identifying $M$ with its image in Euclidean space, we choose a neighborhood $U$ of $M$ in $\mathbb{R}^{n+r}$, diffeomorphic to the total space $E(\nu^r)$ of the normal bundle $\nu^r$ to $M$. Let $m,q \geq n$ and $\gamma^m_q$ be the tautological bundle of oriented $m$-planes over $\mathbb{R}^{m+q}$, and let $E(\gamma^m_a)$ be the total space of $\gamma^m_a$, $a=n, q$. Applying the Gauss map for Grassmannians, we obtain \beq U \cong E(\nu^r) \to E(\gamma^r_n) \subseteq E(\gamma^r_q). \eeq We compose this mapping with the canonical mapping $E(\gamma^r_n) \to T(\gamma^r_q)$, and let $B \subseteq T(\gamma^r_q)$ be the smooth $n$-manifold identified with the zero section in the Thom space. We have obtained a map $f_q : U \to T(\gamma^r_q)$ such that $f_q^{-1}(B)=M$. Hence, if $t_q$ is the base point of the Thom space $T(\gamma^r_q)$, it follows from above that $f_q$ is homotopic to a smooth map $\widehat{f}_q$ that is smooth on $T(\gamma^r_q)-\{ t_q\}$. Thus, we obtain a surjective homomorphism $\pi_{n+r}(T(\gamma^r_q),t_q) \to \Omega_n$. Taking a direct limit on $q$ the claim follows.
\end{proof}
\par Using the argument above, we can embed a representative manifold $M \in \Omega_n$ into $\operatorname{MSO}(r)$ for some $r \geq n+2$. By taking the direct limit, we make the construction independent of the embedding \cite{MR0061823} and obtain a \emph{canonical} isomorphism
\beq
 \Omega_n \cong \varinjlim_{r \to \infty} \pi_{n+r}\big(\operatorname{MSO}(r),\infty\big)\,.  
 \eeq 
\par Another crucial result is a theorem by Serre \cite{MR0059548} that asserts that in the range less than or equal to two times the connectivity of a space, 
rational homotopy is the same as rational cohomology. As shown in \cite{MR0054960}, the connectivity of the Thom space $\operatorname{MSO}(r)$ is $(r-1)$. Therefore, 
there are isomorphisms for all $r \geq n+2$ of the form
 \beq  
  \pi_{n+r}\big(\operatorname{MSO}(r),\infty\big) \otimes \mathbb{Q} \overset{\cong}{\longrightarrow} \tilde{H}^{n+r}\big(\operatorname{MSO}(r);\mathbb{Q}\big)\,.
\eeq 
From Equation~(\ref{eqn:ThomIso}) we also have the Thom isomorphisms
\beq  
H^n\big(\operatorname{BSO}(r);\mathbb{Q}\big)  \longrightarrow \tilde{H}^{n+r}\big(\operatorname{MSO}(r); \mathbb{Z}\big) \,.
\eeq 
Thus, for all $n \geq 0$ it follows that
\beqn  
\label{eqn:BSO}
 H^n\big(\operatorname{BSO}(r);\mathbb{Q}\big) \cong \pi_{n+r}\big(\operatorname{MSO}(r),\infty\big) \otimes \mathbb{Q} \cong \Omega_n \otimes \mathbb{Q}. 
\eeqn 
Equation~(\ref{eqn:BSO}) together with $\operatorname{BSO}= \lim_{r \to \infty} \operatorname{BSO}(r)$ yields -- after taking the appropriate limits on $n$ -- an isomorphism 
\beqn 
\label{eqn:BSOiso}
\Omega_* \otimes \mathbb{Q} \overset{\cong}{\longrightarrow} H^*\big(\operatorname{BSO}; \mathbb{Q}\big)\,.
\eeqn
\par Equation~(\ref{eqn:BSOiso}) shows that the study of genera is equivalent to studying rational homomorphisms $H^*(\operatorname{BSO};\mathbb{Q}) \to \mathbb{Q}$. However, the cohomology ring $H^*(\operatorname{BSO}; \mathbb{Q})=\mathbb{Q}[\mathbf{p}_1,\mathbf{p}_2, \dots]$ is easy to understand: it is a polynomial ring over $\mathbb{Q}$ generated by the Pontrjagin classes $\{\mathbf{p}_i\}$ of the universal classifying bundle $\operatorname{ESO} \to \operatorname{BSO}$ \cite{MR0061823}.  The ring isomorphism $ \Omega_* \otimes \mathbb{Q} \to H^*(\operatorname{BSO}; \mathbb{Q})$ identifies for all $1 \le i \le k$ the Pontrjagin classes of a $4k$-manifold $M$ by pull-back $p_i(M)= f^*\mathbf{p}_i$ under the classifying map $f: M \to \operatorname{BSO}$ for the tangent bundle $TM$. 
\par Thus, the structure of the ring $\Omega_* \otimes \mathbb{Q}$ and its dual space of rational homomorphisms $\hom_\mathbb{Q}( \Omega_* \otimes \mathbb{Q}; \mathbb{Q})$ is now easily understood: the latter consists of sequences of homogeneous polynomials in the Pontrjagin classes with coefficients in $\mathbb{Q}$. In fact, given a genus $\psi  \in  \hom_\mathbb{Q} (\Omega_* \otimes \mathbb{Q}; \mathbb{Q})$, there exists a homogeneous polynomial $L_k \in \mathbb{Q}[\mathbf{p}_1, \dots,\mathbf{p}_k]$ for every degree $4k$ such that for any manifold $M \in \Omega_{4k}$ we have
\beqn
\label{eqn:Lpoly} 
 \psi([M])=L_k(p_1,\dots,p_k)[M] \in \mathbb{Q} \,,
\eeqn 
where the evaluation is carried out according to Equation~(\ref{eqn1}) and Equation~(\ref{eqn2}) \cite{MR0061823}. Hence, the homomorphism $\psi$ is associated to a sequence of homogeneous polynomials $\{L_1,L_2,\dots\} \subseteq \mathbb{Q}[\mathbf{p}_1,\mathbf{p}_2, \dots]$, where $L_k$ is homogeneous of degree $4k$. 
\par  We summarize the results of this section in the following:
\begin{theorem}[Rational cobordism ring]
\label{thm:homBSO}
The rational oriented cobordism ring $\Omega_* \otimes \mathbb{Q}$ is isomorphic to the cohomology ring $H^*(\operatorname{BSO};\mathbb{Q})$. In particular, each element of $\hom_\mathbb{Q}( \Omega_* \otimes \mathbb{Q}; \mathbb{Q})$ is determined by a sequence of homogeneous polynomials in the Pontrjagin classes.
\end{theorem}
The rich structure of $H^*(\operatorname{BSO}; \mathbb{Q})$ also tells us how to find a suitable sequence of generators for $\Omega_* \otimes \mathbb{Q}$. Such a sequence is given by the
cobordism classes of the even dimensional complex projective spaces (thought of as real $4k$-manifolds) \cite{MR0061823}. Thus, each homomorphism is completely determined by its values on the even dimensional complex projective spaces, i.e., the generators of $\Omega_* \otimes \mathbb{Q}$.
\subsection{The Hirzebruch \texorpdfstring{$L$-genus}{L-genus}}
\label{section:hirzebruchsignature} We now have the perspective to characterize the signature in terms of the Pontrjagin numbers. 
Theorem~\ref{thm:sign} asserts that the signature is a cobordism invariant $\operatorname{sign}( \cdot ) \in \hom_\mathbb{Q}(
\Omega_* \otimes \mathbb{Q}; \mathbb{Q})$.  Therefore, there must be a collection of polynomials $\{L_k(p_1,\dots, p_k)\}_{k \in \mathbb{N}}$ such that for any smooth, compact, oriented
manifold of dimension $4k$ the signature is
\beq
 \operatorname{sign}(M) = L_k(p_1, p_2, \dots,p_k)\lbrack M \rbrack \,,
\eeq
where the evaluation is carried out according to Equation~(\ref{eqn1}) and Equation~(\ref{eqn2}). The right hand side is called the \emph{Hirzebruch $L$-genus} \cite{MR0063670}. By Theorem~\ref{thm:homBSO}, the polynomials $L_k$ have the general form
\beqn
\label{eqn:Lpolynom}
 L_k= L_k(p_1, p_2, \dots, p_k) = \sum_{|(i)|=k} \ell^{(i)}_k \; p_{i_1} \cdots p_{i_m} \,.
\eeqn
\par Let us explain how to compute a few of these $L$-polynomials:
\begin{example}
Theorem~\ref{thm:sign} asserts that the signature takes the value $1$ on all even dimensional
complex projective spaces.  Thus, we have the sequence of equations 
\beq 
\begin{array}{rclcl} 1 & = & \operatorname{sign}(\mathbb{C} P^{2}) &
= & \ell^{(1)}_1 \; p_{1}\lbrack\mathbb{C} P^{2}\rbrack \\ 1 & = &
\operatorname{sign}(\mathbb{C} P^{4}) & = & \ell^{(1,1)}_2 \; p_{1}^2\lbrack\mathbb{C}
P^{4}\rbrack + \ell^{(2)}_2 \; p_{2} \lbrack \mathbb{C} P^{4} \rbrack \,,\\ 
& & & \vdots \\
\end{array} 
\eeq
and relations arising from the multiplicativity. 
The total Pontrjagin class of the even complex projective spaces are given by $p(\mathbb{C} P^{2k})=(1+\mathrm{H}^2)^{2k+1}$ with $\mathrm{H}^{2k+1}\equiv 0$. Thus, $p_1(\mathbb{C} P^{2})=3\mathrm{H}^2$ and using Equation~(\ref{eqn:normalization}) we obtain $\ell^{(1)}_1 =1/3$. Compare this result with Example \ref{example:signature}.
\par To find the polynomial $L_2$, we observe that generators of $\Omega_8$ are given by $\mathbb{C} P^4$ and $\mathbb{C} P^2 \times \mathbb{C} P^2$, with total Pontrjagin classes 
\beq 
p(\mathbb{C} P^4)  = 1 + 5\mathrm{H}^2 + 10\mathrm{H}^4 \, \qquad
p(\mathbb{C} P^2 \times \mathbb{C} P^2) = (1 + 3\mathrm{H}_1^2)(1 + 3\mathrm{H}_2^2) \,,
\eeq  
where $\mathrm{H}_1^2,\mathrm{H}_2^2$ are the generators of the respective copies of $H^4(\mathbb{C} P^2;\mathbb{Z})$. Then we have the following system of linear equations
\beq 
1 & = & \ell^{(1,1)}_2 \; p_{1}^2\lbrack\mathbb{C} P^{4}\rbrack + \ell^{(2)}_2 \; p_{2} \lbrack \mathbb{C} P^{4} \rbrack \\ 
1 & = & \ell^{(1,1)}_2 \; p_{1}^2 \lbrack \mathbb{C} P^2 \times \mathbb{C} P^2 \rbrack + \ell^{(2)}_2 \; p_{2} \lbrack \mathbb{C} P^2 \times \mathbb{C} P^2 \rbrack  \,,
\eeq 
which evaluates to the system 
\beq 
1  =  25\ell^{(1,1)}_2 + 10\ell^{(2)}_2 \,, \qquad
1  =  18\ell^{(1,1)}_2 + 9\ell^{(2)}_2. 
\eeq
The solution is $\ell^{(1,1)}_2=-\frac{1}{45}, \ell^{(2)}_2=\frac{7}{45}$. Thus we can conclude
 \beq
 L_1 =  \frac{1}{3} p_1 \,,\qquad
 L_2 =  \frac{1}{45}(7p_2-p_1^2) \,.
 \eeq
\end{example}
\par In general, the homogeneous polynomials $L_k$ are found by solving a system of linear equations on the generators of $\Omega_{4k}$. These are the products $\mathbb{C} P^{2i_1} \times \cdots \times \mathbb{C} P^{2i_m}$ where $(i_1,\dots, i_m)$ ranges over all partitions of $k$. The calculation can be formalized with the help of the so called \emph{multiplicative sequences}. Results in \cite{MR0063670} show that the coefficients of all polynomials $\{L_k\}_{k \in \mathbb{N}}$ can be efficiently stored in a formal power series $Q(z)=\sum_{i=0}^{\infty} b_i z^i$ with $b_0=1$ that satisfies $Q(z w)=Q(z)Q(w)$. 
\par To find the multiplicative sequence from a formal power series $Q(z)$, first define the \emph{Pontrjagin roots} by the formal factorization of the total Pontrjagin class $\sum_{i=1}^k p_i z^i = \prod_{i=1}^k (1 + t_i z)$, that is, consider the Pontrjagin classes the elementary symmetric functions in variables $t_1, \dots, t_k$ of degree $2$.  Thus, given a $4k$-manifold $M$, we have the formal factorization of the total Pontrjagin class 
\beq 
 p(M) = 1 + p_1 + p_2 + \dots p_k = (1 + t_1)(1+t_2) \cdots (1 + t_k), 
\eeq 
whence
\beq 
p_1  =  t_1 + \dots + t_k \,, \qquad
p_2  =  t_1t_2 + t_1 t_3 + \dots + t_{k-1}t_k \,, \quad \dots   
\eeq 
Then from the multiplicative property of $Q$ we obtain
\beq 
Q\big(p(M)\big)=Q(1+t_1) \, Q(1+t_2) \cdots Q(1+t_k). 
\eeq
Putting the formal variable $z$ back into the equation, it turns out that the equation
 \beq 
  \sum_{i=0}^k L_i(p_1, \dots, p_i) \, z^i + O(z^{k+1}) = \prod_{i=1}^r Q(1+t_i z)
 \eeq
 can be used to calculate the polynomial $L_k$ recursively. 
 The following result is crucial:
 \begin{lemma} [\cites{MR0063670, MR0440554}]
 \label{lem:FPS}
The coefficient  of  $p_{i_1}\cdots p_{i_m}$ in $L_k$ in Equation~(\ref{eqn:Lpolynom}) corresponding to the partition $(i) = (i_1,\dots, i_m)$ with $i_1 \geq \dots \geq i_m$ and $\sum i_a =k$ is calculated as follows: let $s_{(i)}(p_1,\dots, p_k)$ be the unique polynomial such that 
 \beq 
 s_{(i)}(p_1,\dots,p_k)= \sum t_1^{i_1}\cdots t_m^{i_m}.
 \eeq 
Then the coefficient of $p_{i_1}\cdots p_{i_m}$ in $L_k$ is $s_{(i)}(b_1,\dots ,b_k)$, where $Q(z)=\sum_{i=0}^{\infty} b_i z^i$ is the formal power series of the genus. Furthermore, the polynomial $L_k$ is given by 
\beq 
 L_k(p_1,\dots,p_k) = \sum_{(i)} s_{(i)}(b_1,\dots,b_k) \, p_{(i)} \,,
\eeq where the sum is over all partitions $(i)$ of $k$ and $p_{(i)} = p_{i_1}\cdots p_{i_m}$. 
\end{lemma}
 \par Lemma~\ref{lem:FPS} allows one to compute the multiplicative sequence of polynomials $\{L_k(p_1,\dots, p_k)\}_{k \in \mathbb{N}}$ from a formal power series $Q(z)$, and conversely, from a given formal power series $Q(z)$, the multiplicative sequence of polynomials $\{L_k(p_1,\dots, p_k)\}_{k \in \mathbb{N}}$. This leads to the following \cite{MR0063670}:
 \begin{theorem}[Hirzebruch signature theorem]
 \label{thm:Hirzebruch}
 Let $\{L_k(p_1,\dots, p_k)\}_{k \in \mathbb{N}}$ be the multiplicative sequence of polynomials corresponding to the formal power series 
\beq 
 Q(z)=\frac{\sqrt{z}}{\tanh{\sqrt{z}}} = 1 + \frac{1}{3}z - \frac{1}{45}z^2 + \frac{2}{945}z^3 - \frac{1}{4725}z^4 + \dots .
 \eeq 
For any smooth, compact, oriented manifold $M$ of dimension $4k$ the signature is
\beq
 \operatorname{sign}(M) = L_k(p_1, p_2, \dots,p_k)\lbrack M \rbrack \,.
\eeq
 \end{theorem}
 \par Let us illustrate the use of the theorem in the following:
\begin{example} 
A computation with the first two symmetric polynomials gives
\beq 
 s_2(p_1, p_2)=p_1^2-p_2 =t_1^2+t_2^2 \,, \qquad
 s_{1,1}(p_1, p_2)  =  p_2 = t_1t_2 \,.
 \eeq 
 Hence, we have 
 \beq 
  L_2 & = &  s_2(1/3,-1/45)p_2 + s_{1,1}(1/3,-1/45)p_1^2 
  \ = \ \frac{1}{45}(7p_2-p_1^2) \,,
 \eeq 
 which confirms the computation for $L_2$ above. Similarly, one checks that 
 \beq 
 s_3(p_1,p_2,p_3) &=& p_1^3-3p_1p_2+3 p_3 = t_1^3+t_2^3+t_3^3 \,,\\ 
 s_{2,1}(p_1,p_2,p_3) &=& p_1p_2 - 3p_3 = t_1^2t_2 + t_1t_2^2 + \dots \,,\\  
 s_{1,1,1}(p_1,p_2,p_3) &=& p_3 = t_1t_2t_3 \,,
 \eeq 
and one obtains
 \beq 
  L_3 &=&  \frac{1}{945}\big(62p_3-13p_2p_1+2p_1^3\big) \,.
\eeq
For $\mathbb{C} P^6$ we have $p_1=7 \mathrm{H}^2$, $p_2=21 \mathrm{H}^4$, $p_3=35 \mathrm{H}^6$ so that $L_3 [ \mathbb{C} P^6]=1$ when using Equation~(\ref{eqn:normalization}).
\end{example} 
\par We close this section with several remarks on the Hirzebruch signature theorem:
\begin{remark}
The signature is an integer, so the Hirzebruch signature theorem imposes
non-trivial integrality constraints on the rational combinations of the Pontrjagin numbers determined by the $L$-polynomials. 
\end{remark}
\begin{remark}
The signature is an oriented homotopy invariant, whereas the L-polynomials are
expressed in terms of Pontrjagin classes that rely heavily on
the tangent bundle, and thus on the smooth structure, and
orientation. This fact was used by Milnor to detect inequivalent differentiable structures on the $7$-sphere \cite{MR0082103}. According to a theorem of Kahn \cite{MR0172306}, the $L$-genus is (up to a rational multiple) the only rational linear combination of the Pontrjagin numbers that is an oriented homotopy invariant.
\end{remark}
\begin{remark}
\label{rem:AS}
Combining Theorem~\ref{thm:analytic_sign} and Theorem~\ref{thm:Hirzebruch} implies that for a compact, oriented Riemannian manifold $M$ of dimension $4k$, the index of the chiral Dirac operator $\slashed{D}$ is given by 
\beqn
\label{eqn:AS}
\mathrm{ind}(\slashed{D}) =  L_k(p_1,\dots,p_k)[M] = \int_M L_k \,.
\eeqn
This statement is a special case of the celebrated Atiyah-Singer index theorem, which proves that for an elliptic differential operator on a compact manifold, the analytical index equals the topological index defined in terms of characteristic classes \cite{MR157392}.
\end{remark}
\begin{remark}
The signature complex can be twisted by a complex vector bundle $\xi \to M$ of rank $r$ \cites{MR769356, MR598586}. The coupling of the complexified signature complex to the vector bundle $\xi$ yields a twisted chiral Dirac operator 
\beq 
 \slashed{D}^\xi :  C_+^\infty(M, \Lambda^* \otimes \xi) \to C_-^\infty(M, \Lambda^* \otimes \xi) \,.
\eeq 
For simplicity, we will assume $c_1(\xi)=0$ and that the dimension of the manifold $M$ is $4$. Then the effect of the twisting on the index is as follows:
\beqn
\label{eqn:AStwist}
\mathrm{ind}(\slashed{D}^\xi) = r \cdot \mathrm{ind}(\slashed{D}) - \int_M c_2(\xi) \,,
\eeqn
where $c_2(\xi)$ is the second Chern class of the complex bundle $\xi \to M$ of rank $r$.
\end{remark}
\section{The determinant line bundle and the RRGQ formula}
\label{sec:det}
In this section we consider the \emph{family} of (complexified) signature operators $\{\slashed{D}_g\}$ -- where we denote the operators as chiral Dirac operators defined in Section~\ref{ssec:index} -- acting on sections of suitable bundles over a fixed even-dimensional manifold $M$, with the operators being parameterized by (conformal classes of) Riemannian metrics $g$. We will focus on the simplest nontrivial case, when $M$ is a flat two-torus, and $g$ varies over the moduli space $\mathfrak{M}$ of flat metrics. 
\par Given a two-torus $M$ equipped with a complex structure $\tau \in \mathbb{H}$, we identify $M=\mathbb{C} / \langle 1, \tau\rangle$ where  $\langle 1, \tau\rangle$ is a rank-two lattice in $\mathbb{C}$. This is done by identifying opposite edges of each parallelogram spanned by $1$ and $\tau$ in the lattice to obtain $M=\mathbb{C} / \langle 1, \tau\rangle$. We endow $M$ with a compatible flat torus metric $g$ that descends from the flat metric on $\mathbb{C}$. Following Section~\ref{ssec:index} we define an operator $D=d + \delta$ on the even dimensional manifold $M$. The complexified signature operator, written as a chiral Dirac operator $\slashed{D}_\tau$, is obtained from $D$ in Equation~(\ref{eqn:chiral}). The subscript $\tau$ reminds us of the dependence on the metric, and hence choice of complex structure, in the definition of $\delta$. We ask: as $\tau$ varies, is there anomalous behavior in the family of chiral Dirac operators (acting on complexified differential forms)
\beq 
 \slashed{D}_{\tau} :  C_+^\infty(M, \Lambda^*) \to C_-^\infty(M, \Lambda^*) \,?
\eeq 
\par A more refined question is the following: consider a Jacobian elliptic surface, given as a holomorphic family of elliptic curves $\mathcal{E}_t$ (some of them singular) over the complex projective line $\mathbb{C} P^1 \ni t$; the exact definition and the relation between $\tau$ and $t$ will be discussed below. The numerical index of the chiral Dirac operator $\slashed{D}_t$ vanishes since the signature of each smooth fiber $\mathcal{E}_t$, is zero -- a smooth torus forms the boundary of a smooth compact 3-manifold. Thus, the numerical index does not yield any interesting geometric information. 
\par As a more refined invariant, we are interested in the \emph{determinant line bundle} $\operatorname{\mathbf{Det}}{\slashed{D}} \to \mathbb{C} P^1$ associated with the family of Dirac operators $\slashed{D}_t$, and its first Chern class $c_1(\operatorname{\mathbf{Det}}{\slashed{D}})$.  As a generalized cohomology class, this class will measure the so-called \emph{local and global anomaly}, revealing crucial information about the family of operators $\slashed{D}_t$, and it is this quantity that we will be studying for the remainder of this article.  
\subsection{The determinant line bundle}
For an elliptic differential operator $\slashed{D}$ on a compact manifold $M$, the kernel and cokernel are finite dimensional vector spaces; thus, one can define the one-dimensional vector space
\beqn
\label{eqn:det_bdl}
\operatorname{\mathbf{Det}}{\slashed{D}} = (\Lambda^{\mathrm{max}} \ker{\slashed{D}})^* \otimes (\Lambda^{\mathrm{max}}\operatorname{coker}{\slashed{D}}). 
\eeqn
The vector space $\operatorname{\mathbf{Det}}{\slashed{D}}$ is the dual of the maximal exterior power of the index $\operatorname{\mathbf{Ind}}{\slashed{D}}$ of $D$, that is,  the formal difference of $\ker{\slashed{D}}$ and $\operatorname{coker}{\slashed{D}}$, given by
\beq 
\operatorname{\mathbf{Ind}}{\slashed{D}} =  \ker{\slashed{D}} - \operatorname{coker}{\slashed{D}}. 
\eeq
\par Let $\pi : Z \to B$ be a smooth fiber bundle with compact fibers, and $E$ and $F$ be smooth vector bundles on $Z$, with a smooth family of elliptic operators $\slashed{D} = (\slashed{D}_t)_{t \in B}$ acting on the fibers $\pi^{-1}(t)$ as
\beq
 \slashed{D}_t : C^{\infty}(\pi^{-1}(t),E) \to C^{\infty}(\pi^{-1}(t),F) \,.
\eeq 
Even though the dimensions of the kernel and cokernel of $\slashed{D}_t$ can jump, it turns out that there is a canonical structure of a differentiable line bundle on the family of one-dimensional vector spaces $\{ \operatorname{\mathbf{Det}}{\slashed{D}_t}\}_{t \in B}$ \cite{MR915611}. Equivalently, we can say that the one-dimensional vector spaces $\{ \operatorname{\mathbf{Det}}{\slashed{D}_t}\}_{t \in B}$ patch together to form a line bundle $\operatorname{\mathbf{Det}}{\slashed{D}} \to B$. We remark that the formal difference used to define the index $\operatorname{\mathbf{Ind}}{\slashed{D}}$ also makes sense in the context of $K$-theory, i.e., as a well-defined index bundle, an element of the $K$-theory group $K(B)$ \cite{MR0279833}.
\par We choose a smooth family of Riemannian metrics on the fibers $\pi^{-1}(t)$, and smooth Hermitian metrics on the bundles $E$ and $F$. Then, for any $t \in B$, the adjoint operator $\slashed{D}^\dagger_t$ is well defined, and the vector spaces $\ker{\slashed{D}_t}$ and $\ker{\slashed{D}^\dagger_t}$ have natural $L^2$ metrics, which in turn define a metric $\Vert  \cdot \Vert _{L^2}$ on $\operatorname{\mathbf{Det}}{\slashed{D}_t}$.  However, because of the jumps in the dimensions of the kernel and cokernel, this does not define a \emph{smooth} metric on the line bundle $\operatorname{\mathbf{Det}}{\slashed{D}} \to B$. Instead, a way to assign a smooth metric is the \emph{Quillen metric}, which uses the analytic torsion of the family $(\slashed{D}_t)_{t \in B}$ to smooth out the effect these jumps \cite{MR915611}. The Quillen metric on $\operatorname{\mathbf{Det}}{\slashed{D}} \to B$ is given by
\beq 
 \Vert  \cdot \Vert _Q = (\sideset{}{'}\det \slashed{D}_t^\dagger \slashed{D}_t)^{\frac{1}{2}} \, \Vert  \cdot \Vert _{L^2}, 
\eeq 
where $\sideset{}{'}\det \slashed{D}_t^\dagger \slashed{D}_t$ is the analytic torsion of the family of operators $(\slashed{D}_t)_{t\in B}$. It is crucial that the operators $\Delta_t= \slashed{D}_t^\dagger \slashed{D}_t$ form a family of positive, self-adjoint operators acting on sections of a vector bundle over compact manifolds.
\par The \emph{analytic torsion}, or \emph{regularized determinant} of a positive, self-adjoint operator $\Delta$ acting on sections of a vector bundle over a compact manifold is defined as follows: by the hypotheses on $\Delta$, it follows that the operator $\Delta$ has a pure point spectrum of eigenvalues, denoted by $\{\lambda_j\} \subset \mathbb{R}_{\ge 0}$. If there were only finitely many eigenvalues, then we could write down the identity
 \beq \frac{d}{ds}\left(\sum_{\lambda_j \neq 0} \lambda_j^{-s}\right)\Bigg{|}_{s=0} = - \sum_{\lambda_j \neq 0} \log \lambda_j,
 \eeq 
and compute the product of eigenvalues, i.e., the determinant of the operator $\Delta$, as 
\beq
 \prod_{\lambda_j\neq 0} \lambda_j = \exp\left( \sum_{\lambda_j \neq 0} \log \lambda_j \right).
\eeq 
Since there are infinitely many eigenvalues, we define a $\zeta$-function instead, given by 
\beq \zeta_\Delta(s) = \sum_{\lambda_j \neq 0} \lambda_j^{-s} \, .
\eeq 
It turns that $\zeta_\Delta(s)$ is a well defined holomorphic function for $\re(s) \gg 0$ with a meromorphic continuation to $\mathbb{C}$ such that $s=0$ is not a pole of $\zeta_\Delta$ \cites{MR915611, MR0383463}. Then, the regularized determinant $\sideset{}{'}\det \Delta$ of $\Delta$ is defined as 
\beqn 
\label{eqn:regdet}
\sideset{}{'}\det \Delta = \exp\big(- \zeta_\Delta'(0)\big) \, ,
\eeqn 
where the prime on $\sideset{}{'}\det \Delta$ indicates that the zero eigenvalue has been dropped.

\subsection{Jacobian elliptic surfaces}
\label{section:ellipticsurfaces}
Let $Z$ be a compact connected complex manifold of dimension two. This means that at every point of $Z$ there are two local complex coordinates and the change of coordinate charts are holomorphic maps. We call $Z$ an \emph{algebraic surface}  if and only if the field of global meromorphic functions of $Z$ allows us to separate all points and tangents. That is, (i) for every two points there is a global meromorphic function on $Z$ that has different values at the two points, and (ii) for every point there are global meromorphic functions on $Z$ that define the local coordinates around the given point. An \emph{elliptic fibration} over $\mathbb{C} P^1$ on $Z$ is a holomorphic map $\pi : Z \to \mathbb{C} P^1$ such that the general fiber of $\pi^{-1}(t)$ is a smooth curve of genus one with $t \in \mathbb{C} P^1$. An \emph{elliptic surface} is an algebraic surface with a given elliptic fibration. Since a curve of genus one, once a point has been chosen, is isomorphic to its Jacobian, i.e., an elliptic curve, we call an elliptic surface a \emph{Jacobian elliptic surface} if it admits a section that equips each fiber with a smooth base point. In this way, each smooth fiber is an abelian group and the base point serves as the origin of the group law.

To each Jacobian elliptic fibration $\pi: Z \to B\cong \mathbb{C} P^1$ there is an associated Weierstrass model obtained by contracting all components of fibers not meeting the section. If we choose $t \in \mathbb{C}$ as a local affine coordinate on $\mathbb{C} P^1$ and $(x,y)$ as local coordinates of the elliptic fibers, we can write the Weierstrass model of an elliptic curve as
\begin{equation}
\label{Weierstrass}
 y^2 = 4 \, x^3 - g_2(t) \, x - g_3(t) \;,
\end{equation}
where $g_2$ and $g_3$ are polynomials in $t$ of degree four and six, or, eight and twelve  if $Z$ is a rational surface, i.e., birational to $\mathbb{C} P^2$, or a $K3$ surface, respectively.  Since we contracted all components of the fibers not meeting the zero-section, the total space of Equation~(\ref{Weierstrass}) is always singular  with only rational double point singularities and irreducible fibers, and $Z$ is the minimal desingularization. The discriminant $\Delta=g_2^3-27 \, g_3^2$ vanishes where the fibers of Equation~(\ref{Weierstrass}) are singular curves.  It follows that if the degree of the discriminant $\Delta$ is a polynomial of degree $12$ (or $24$), considered as a homogeneous polynomial on $\mathbb{C} P^1$, then the minimal desingularization of the total space of Equation~(\ref{Weierstrass}) is a rational elliptic surface (resp.~\! $K3$ surface). 
\par In his seminal paper \cites{MR0165541}, Kodaira realized the importance of elliptic surfaces and proved a complete classification for the possible singular fibers of the Weierstrass models. Each possible singular fiber over a point $t_0$ with $\Delta(t_0)=0$ is uniquely characterized in terms of the vanishing degrees of $g_2, g_3, \Delta$ as $t$ approaches $t_0$. The classification encompasses two infinite families $(I_n, I_n^*, n \ge0)$ and six exceptional cases $(II, III, IV, II^*, III^*, IV^*)$. Note that the vanishing degrees of $g_2$ and $g_3$ are always less than or equal to three and five, respectively,
as otherwise the singularity of Equation~(\ref{Weierstrass}) is not a rational double point.
\par Closely related is the $j$-function, a holomorphic map $\mathsf{j}: B \to \mathbb{C} P^1$ that can be computed from a Weierstrass model using the formula
\beqn
\label{eqn:jinv} 
\mathsf{j} = \frac{g_2^3}{\Delta}. 
\eeqn 
Every smooth elliptic fiber $\mathcal{E}_t=\pi^{-1}(t)$ is a complex torus, and thus can be identified with a rank-two lattice $\Lambda$ to obtain $\mathcal{E}_t \cong \mathbb{C}/\Lambda$. However, multiplying the lattice $\Lambda$ by certain complex numbers (which corresponds to rotating and scaling the lattice) can preserve the isomorphism class of an elliptic curve. Hence we can always arrange for the lattice to be generated by $1$ and some complex number $\tau \in \mathbb{H}$ in the upper half plane; we write $\Lambda_\tau = \langle 1, \tau \rangle$. Moreover, two $\tau$-parameters $\tau_1$ and $\tau_2$ in $\mathbb{H}$ belong to isomorphic elliptic curves if and only if 
\beq 
\tau_2=\frac{a\tau_1+b}{c\tau_1+d}  \hspace{.5in} \mathrm{for \; some} \; \begin{pmatrix} a & b \\ c & d \end{pmatrix} \in \operatorname{PSL}(2,\mathbb{Z})\,, 
\eeq 
where the \emph{modular group} $\operatorname{PSL}(2,\mathbb{Z})$ acts (projectively) on $\mathbb{H}$. It can be shown that the action of the modular group on the fundamental domain \beq \mathcal{D}= \left\{ \tau \in \mathbb{H} \; \bigg{|} \; \re(\tau)\leq \frac{1}{2}, |\tau| \geq 1 \right\} \eeq generates $\mathbb{H}$ \cite{MR0344216} such that the moduli space of isomorphism classes of elliptic curves is realized as $\mathbb{H} / \operatorname{PSL}(2,\mathbb{Z}) \cong \mathcal{D}$. The one point compactification$\mathcal{D}\cup \{\infty\}=\mathbb{C} P^1$ is also called the \emph{coarse moduli space} of elliptic curves. The ``corners'' of $\mathcal{D}$ are of fundamental importance: these are the numbers $\rho = e^{2 \pi i /3}, i,$ and $-\overbar{\rho}= e^{\pi i/3}$.
\par It can be shown that under the identification $\mathcal{E}_t \cong \mathbb{C}/\Lambda_\tau$ the discriminant $\Delta$ becomes a modular form of weight twelve, and $g_2$ one of weight four, so that its third power is also of weight twelve. For example, we may express the discriminant as 
\beqn 
\label{eqn:discr}
\Delta_\tau = e^{2\pi i \tau} \prod_{r=1}^{\infty}\left(1-e^{2\pi i \tau r}\right)^{24}.
\eeqn 
Thus, the quotient in Equation~(\ref{eqn:jinv}) is a \emph{modular function} of weight zero, in particular it defines a holomorphic function $j: \mathbb{H} \to \mathbb{C} P^1$ invariant under the action of $\operatorname{PSL}(2,\mathbb{Z})$ such that  for every smooth elliptic fiber $\mathcal{E}_t \cong \mathbb{C}/\Lambda_\tau$ we have $\mathsf{j}(t)=j(\tau)$ and $\Delta(t)=\Delta_\tau$. A more careful examination of the behavior at the corners yields $j(\rho)=0$, $j(i)=1$, and $j(-\overbar{\rho})=\infty$. 
\par As an example, we consider the Jacobian elliptic surface $\pi: S \to \mathbb{C} P^1$ given by the Weierstrass model
\beqn
\label{eqn:Ubdle}
 y^2  = 4x^3 - 27t(t-1)^3x -27t(t-1)^5  \,,
\eeqn
where $t$ is  the affine coordinate on the base curve with $[t:1] \in \mathbb{C} P^1$. This family was considered in \cite{MR584462}, and it follows from definitions that
\beqn
\label{eqn:Discriminant}
 \Delta(t) = 27^3t^2(t-1)^9 \,, \qquad \mathsf{j}(t) =  \frac{g_2^3}{\Delta} = t \,.
\eeqn
The fibers over $t=0, 1, \infty$ are singular; in the language of Kodaira's classification result, the singular fibers over $t=0,1,\infty$ correspond to fibers of type $II, III^*,$ and $I_1$, respectively \cites{MR0165541}. The total space $\overbar{S}$ of Equation~(\ref{eqn:Ubdle}) is singular, and its minimal desingularization is the rational elliptic surface $S$. It is a rational Jacobian elliptic surface whose $j$-function is the coordinate on the base curve itself, it is also called the \emph{universal family of elliptic curves}. We call the base curve the $j$-line and denote it by $J= \mathbb{C} P^1$.
\par Notice that if one has a local affine coordinate $t$ on a base curve $B$, and one replaces $g_2$ by $g_2 t^2$ and $g_3$ by $g_3 t^3$ in Equation~(\ref{Weierstrass}), the $j$-function in Equation~(\ref{eqn:jinv}) is left invariant. This operation, called a \emph{quadratic twist}, does change the nature of the singular fibers: it switches $I_n$ and $I_n^*$ fibers, as well as $II$ and $IV^*$, $IV$ and $II^*$, and $III$ and $III^*$. Therefore, the $j$-function does not determine the elliptic surface, not even locally. However, the quadratic twist is the only way that two Jacobian elliptic surfaces can have the same $j$-function, and conversely, a Jacobian elliptic fibration is uniquely determined by the $j$-function up to quadratic twist. Moreover, the canonical holomorphic map $\mathsf{j}: B \to J$ in Equation~(\ref{eqn:jinv}) can be lifted to a (rational) map between the elliptic surfaces $Z$ and $S$ themselves. Thus, we have the following:
\begin{corollary}
\label{cor:Jmap}
Let $\pi : Z \to B=\mathbb{C} P^1$ be a Jacobian elliptic surface. There is a canonical holomorphic map $\mathsf{j}: B \to J$ that uniquely determines the Jacobian elliptic surface $Z$ up to quadratic twist. Moreover, there is an induced rational map $Z \dasharrow S$ between the total spaces. The map $\mathsf{j}$ has degree $1$ or $2$ if $Z$ is a rational or a K3 surface, respectively.
\end{corollary}
Let $\pi: Z \to B\cong \mathbb{C} P^1$ be a Jacobian elliptic surface introduced in Section~\ref{section:ellipticsurfaces}. Using the adjunction formula, we define the relative canonical bundle $K_{Z | B}$ of the elliptic surface in terms of the canonical bundles of $Z$ and $B$, respectively, by writing
\beq
   K_{Z | B} = K_Z \otimes (\pi^*K_B)^{-1}  \,.
\eeq
The bundle $K_{Z | B}$ can be identified with the line bundle of vertical $(1,0)$-forms of the fibration $\pi: Z \to B$.  Using the push-forward operation $\pi_* K_{Z | B}$ in algebraic geometry, we obtain a bundle $\mathcal{K}=\pi_* K_{Z | B} \to B$. We have the following:
\begin{lemma}
\label{lem:relat_canonical}
On a Jacobian elliptic surface $Z \to B\cong \mathbb{C} P^1$ given by Equation~(\ref{Weierstrass}) we have $\mathcal{K}=\pi_* K_{Z | B} \cong \mathcal{O}(n)$ with $n=1$ if $Z$ is rational and $n=2$ if $Z$ is a K3 surface. 
\end{lemma}
\begin{proof}
Over the open set $\Delta(t) \not =0$ a canonical section of $K_{Z | B}$ is given by the holomorphic one-form $dx/y$.  The proof follows from the fact that a change of coordinates in Equation~(\ref{Weierstrass}) is given by $(t, x, y) \mapsto (t', x', y') =(1/t, x/t^{2n}, y/t^{3n})$ as $g_2$ and $g_3$ are polynomials of degree $4n$ and $6n$, respectively.  Thus, the holomorphic one-form transform $dx/y \mapsto dx'/y' = t^n \, dx/y$ whence  $\mathcal{K} \cong \mathcal{O}(n)$. 
\end{proof}
We can rephrase the construction of the Weierstrass model in Equation~(\ref{Weierstrass}) in terms of sections of the relative canonical bundle. We will use this point of view later. Let $\mathcal{L} \to B$ be a line bundle on $B \cong \mathbb{C} P^1$, and $g_2$ and $g_3$ sections of $\mathcal{L}^{4}$ and $\mathcal{L}^{6}$, respectively, such that the discriminant $\Delta=g_2^3-27 \, g_3^2$ is a section of $\mathcal{L}^{12}$ not identically zero. Define $\mathbf{P} :=\mathbb{P}( \mathcal{O} \oplus \mathcal{L}^2 \oplus \mathcal{L}^3)$ and let $p:\mathbf{P} \to B$ be the natural projection and $\mathcal{O}_{\mathbf{P}}(1)$ the tautological line bundle. We denote by $X$, $Y$, and $Z$ the sections of $\mathcal{O}_{\mathbf{P}}(1) \otimes \mathcal{L}^{2}$,  $\mathcal{O}_{\mathbf{P}}(1) \otimes \mathcal{L}^{3}$, and $\mathcal{O}_{\mathbf{P}}(1)$, respectively, which correspond to the natural injections of $\mathcal{L}^{2}$,  $\mathcal{L}^{3}$, and $\mathcal{O}$ into $p_* \mathcal{O}_{\mathbf{P}}(1) = \mathcal{O} \oplus \mathcal{L}^{2}  \oplus\mathcal{L}^{3}$. We denote by $W$ the projective variety in $\mathbf{P}$ defined by the equation
\beq
\label{Weierstrass2}
 Y^2  Z = 4  X^3 - g_2(t) \,  X Z^2 - g_3(t) \,   Z^3 \,.
\eeq
Its canonical section $\sigma: \mathbb{C} P \to W$ is given by the point $[X:Y:Z]=[0:1:0]$ such that $\Sigma:=\sigma(\mathbb{C} P^1) \subset W$ is a divisor on $W$, and its \emph{normal bundle} is isomorphic to the fundamental line bundle by $p_* \mathcal{O}_{\mathbf{P}}\big(\!-\!\Sigma\big)\cong \mathcal{L}$. In the affine chart $Z=1$, and $X=x$, $Y=y$  the one-form $dx/y$ is a section of the bundle $\mathcal{L}^{-1}$; hence, the dual of the normal bundle, also called the \emph{conormal bundle}, is precisely the relative canonical bundle introduced above, i.e., $\mathcal{L}^{-1} \cong \mathcal{K}$.
\subsection{The analytic torsion for families of elliptic curves}
\label{section:torsion}
Let us compute the analytic torsion for the Laplacian $\Delta_H=(d+\delta)^2$ from Section~\ref{ssec:index} in the special situation where the even dimensional manifold $M$ is an elliptic curve $\mathcal{E}$, i.e., a flat two-torus equipped with a complex structure. We use the identification $\mathcal{E} \cong \mathbb{C}/\langle 1, \tau \rangle$, the local coordinate $z=x+i y$ on $\mathbb{C}$, and the notation $\del=\del_{z}$ and $\delbar=\del_{\overbar{z}}$. We have the following:
\begin{lemma} 
\label{lem:at}
Let $\mathcal{E} \cong \mathbb{C}/\langle 1, \tau \rangle$ be a smooth elliptic curve endowed with the compatible flat torus metric $g$.  The Laplacian when restricted to $C^\infty(\mathcal{E})$ equals $\Delta_H=-4\del\delbar$, and its analytic torsion is given by 
\beqn
\label{eqn:at} 
 \sideset{}{'}\det  \Delta_H = \left(\frac{\im(\tau)}{2\pi}\right)^2 |\Delta_\tau|^\frac{1}{6},
\eeqn 
where $\Delta_\tau$ is the modular discriminant of the elliptic curve $\mathcal{E}$ in Equation~(\ref{eqn:discr}).
Moreover, the same answer holds for $\Delta_H$ restricted to $C^\infty(\mathcal{E}, T^{*}_\mathbb{C} \mathcal{E}^{(1,0)})$.
\end{lemma}
\begin{proof} On $\mathcal{E}$ the operator $-\Delta_H$ is a positive, self adjoint operator. Endowing $\mathcal{E}$ with a compatible flat metric shows that for a local coordinate $z=x+i y$, we have the scalar Laplacian as
 \beq 
 \Delta_H = -(\del_x^2+\del_y^2)=-4\left(\dfrac{1}{2}(\del_x-i\del_y)\dfrac{1}{2}(\del_x+i\del_y)\right)=-4\del_z\del_{\overbar{z}}=-4\del\delbar.
 \eeq 
The equality holds for the Laplacian acting on $p$-forms since one checks that
\begin{align*}
\Delta_H : \;& f \mapsto -4\del_z\del_{\overbar{z}}f \,, & \qquad f \in C^{\infty}(\mathcal{E})\,,\\
\Delta_H : \;& \phi= (f dz+g d\overbar{z}) \mapsto -4(\del_z\del_{\overbar{z}}f dz+\del_z\del_{\overbar{z}}gd\overbar{z})\,, & \phi \in C^{\infty}(\mathcal{E},\Lambda^1)\,,\\
\Delta_H : \;& \omega=f dz\wedge d\overbar{z} \mapsto -4\del_z\del_{\overbar{z}}f dz\wedge d\overbar{z}\,, &  \omega \in C^{\infty}(\mathcal{E},\Lambda^2) \,.
\end{align*}
\par A function $\varphi$ with a periodicity given by
\beq
 \varphi(x + 1 , y) =   \varphi(x,y) \;, \qquad
 \varphi(x + \re \tau ,y + \im \tau)  =  \varphi(x,y) \;,
\eeq
descends to a well defined function on $\mathcal{E}$. For $n_1, n_2 \in \mathbb{N}$ such a function is given by
\beq
\varphi_{n_1,n_2}(x,y) =  
 \exp{2\pi i  \left( n_1 x + \frac{(n_2 - n_1 \re\tau )}{\im\tau} y \right)} \;.
\eeq
In fact, the functions $\varphi_{n_1,n_2}$ constitute a complete system of eigenfunctions for $\Delta_H$ with the eigenvalues
\beq
 \lambda_{n_1, n_2} = \left(\frac{2\pi}{\im\tau}\right)^2 \; \left\vert n_1 \tau - n_2  \right\vert^2 \;.
\eeq
Notice that we have $\Delta_H \varphi_{n_1,n_2} = \lambda_{n_1, n_2}  \varphi_{n_1,n_2} dz \wedge d\overbar{z}$ and then use the K\"ahler form to identify $dz \wedge d\overbar{z}$ with $1$.  We define a $\zeta$-function $\zeta(s)$ by setting
\begin{eqnarray}
\label{zeta}
 \zeta(s) = \sideset{}{'}\sum_{n_1,n_2} \frac{1}{\left\vert n_1 \tau - n_2  \right\vert^{2s}} \,,
\end{eqnarray}
where the prime indicates that the summation does not  include $n_1=n_2=0$. One checks that $\zeta(s)$ is absolutely convergent for $\re(s)>1$, has a meromorphic extension to $\mathbb{C}$, and 0 is not a pole. It was shown in \cite{MR0383463} that $\zeta(0)=-1$. The regularized determinant of $\Delta_H$ is then given by
\begin{eqnarray*}
 \ln \det \Delta_H & = & 
 - \left[ \dfrac{1}{\left(\frac{2\pi}{  \im\tau }\right)^{2s}}  \; \zeta(s) \right]' 
 =  - \zeta'(0) + \ln\left( \frac{2\pi}{ \im\tau } \right)^2 \; \zeta(0) \,.
\end{eqnarray*}
It was shown in \cite{MR0383463} that $\exp{[ - \zeta'(0)]}  =  \vert \eta(\tau) \vert^4$ using the Kronecker limit formula where the Dedekind $\eta$-function is given by
\begin{align}
 \eta(\tau)  = e^{\frac{\pi i \, \tau}{12}} \; \prod_{r=1}^\infty \left(1- e^{2\pi i \tau r}\right) \,.  
\end{align}
It follows from Equation~(\ref{eqn:discr}) that 
\beqn
\label{eqn:ators}
 \sideset{}{'}\det \Delta_H = \left( \frac{\im(\tau)}{2\pi} \right)^2 |\Delta_\tau|^{\frac{1}{6}} \,.
\eeqn 
A similar argument can be repeated for the sections
\beqn
\label{eqn:eigen}
 \varphi_{n_1,n_2}(z) \, dz \in  C^\infty(\mathcal{E}, T^{*}_\mathbb{C} \mathcal{E}^{(1,0)}) \,.
\eeqn
Applying $\delbar$ we obtain
\beq
  -\left(\frac{\pi}{\im\tau}\right) \; (n_1 \tau - n_2) \; \varphi_{n_1,n_2}(z) \, dz\wedge d\overbar{z} \,.
\eeq
Contracting with the K\"ahler form  and applying $\del$ shows that the sections in Equation~(\ref{eqn:eigen}) form a complete system of eigenfunctions for $\Delta_H$ restricted to the vector space $C^\infty(\mathcal{E}, T^{*}_\mathbb{C} \mathcal{E}^{(1,0)})$.
\end{proof} 
\par It follows from the definitions in Section~\ref{ssec:index} that 
\beq 
\slashed{D}=(d+\delta)P_+= \delbar^{\dagger} + \delbar \, \Big{|}_{C^{\infty}_+(M,\Lambda^*)}, \quad \slashed{D}^{\dagger}=(d+\delta)P_-= \del^{\dagger} + \del \, \Big{|}_{C^{\infty}_-(M,\Lambda^*)}. 
\eeq 
Recall that the kernel and cokernel of the chiral Dirac operator $\slashed{D}$ in Equation~(\ref{eqn:chiral}) are given by self-dual and anti-self-dual generators of the de~\!Rham cohomology classes, respectively, i.e., by elements of $H^*_{\pm}(\mathcal{E};\mathbb{C})=\{ \omega \in H^*(\mathcal{E};\mathbb{C}) \; | \; \alpha(\omega)=\pm \omega \}$. On an elliptic curve $\mathcal{E}$, the forms $1 -  \frac{1}{2} dz \wedge d\overbar{z}$ and $dz$ are self-dual, and similarly the forms $1 + \frac{1}{2}  dz \wedge d\overbar{z}$ and $d\overbar{z}$ are anti-self-dual. Here $dz$ and $d\overbar{z}$ are section of the holomorphic cotangent bundle $T^{*}_\mathbb{C} \mathcal{E}^{(1,0)}=K_\mathcal{E}$ (also called the canonical bundle), and the bundle $T^{*}_\mathbb{C} \mathcal{E}^{(0,1)}= \overbar{K}_\mathcal{E}$, respectively. These (anti-)self-dual differential forms descend to generators of $H^p_+(\mathcal{E};\mathbb{C})$ and $H^p_-(\mathcal{E};\mathbb{C})$, respectively, for $p=0, 1$.
\par We now consider the Jacobian elliptic surface $\pi: S \to J=\mathbb{C} P^1$ from Section~\ref{section:ellipticsurfaces}. The elliptic fiber $\mathcal{E}_t = \pi^{-1}(t)$ given by Equation~(\ref{eqn:Ubdle}) is a smooth elliptic curve with discriminant $\Delta(t)$ for $t \in J^* = J  - \{ 0,1, \infty\}$. The chiral Dirac operator $\slashed{D}_t$ in Equation~(\ref{eqn:chiral}) on each smooth elliptic fiber $\mathcal{E}_t$ is the sum $\slashed{D}_t=\delbar_t \oplus \delbar^\dagger_t$ of the operators
\beq
 \delbar_t: C^\infty\big( \mathcal{E}_t \big) \to C^\infty\big( \mathcal{E}_t, \overbar{K}_{\mathcal{E}_t} \big) \,, \qquad
 \delbar^\dagger_t: C^\infty\big( \mathcal{E}_t, \overbar{K}_{\mathcal{E}_t} \big) \to C^\infty\big( \mathcal{E}_t \big) \,,
\eeq 
where $\overbar{K}_{\mathcal{E}_t}$ denotes the conjugate of the canonical bundle of $\mathcal{E}_t$. For our purposes, it is not necessary to investigate both components of the operator $\slashed{D}_t$; this follows from the factorization of the Laplacian in Lemma~\ref{lem:at}. Thus, we will focus on the operator $\delbar_t$ and its determinant line bundle $\operatorname{\mathbf{Det}}{\delbar} \to J^*$.
\par It turns out that the vector spaces $H^p_+(\mathcal{E}_t;\mathbb{C})$ and $H^p_-(\mathcal{E}_t;\mathbb{C})$, respectively, for $p=0, 1$, patch together to form smooth line bundles over $J^*$ with a smooth Hermitian metrics. The vector spaces $H^0_+(\mathcal{E}_t;\mathbb{C})$ and $H^0_-(\mathcal{E}_t;\mathbb{C})$ each form a trivial line bundle $\underline{\mathbb{C}} \to J^*$. Similarly, the vector spaces $K_{\mathcal{E}_t}$ generated by $dz$ on each elliptic curve $\mathcal{E}_t$ patch together to generate the bundle of vertical $(1,0)$-forms $K_{S | J} \to J$ in Lemma~\ref{lem:relat_canonical}. Since we have described the kernel and cokernel of the chiral Dirac operator $\delbar_t$, it follows from Equation~(\ref{eqn:det_bdl}) that for each $t \in J^* = J  - \{ 0,1, \infty\}$ we have \emph{fiberwise} an identification
\beq
  \operatorname{\mathbf{Det}}{\delbar}_t \cong  \overbar{K}_{S | J} \Big\vert_t  \,.
\eeq
Moreover, the Quillen metric induces a canonical holomorphic structure on the determinant line bundle $\operatorname{\mathbf{Det}}{\delbar} \to J^*$; see \cite{MR915611}.  We have the following:
\begin{proposition}
\label{thm1}
Let $s$ be the canonical holomorphic section of $\operatorname{\mathbf{Det}}{\delbar} \to J^*$. For each $t \in J^* = J  - \{ 0,1, \infty\}$ the Quillen norm of the section $s$ is given by
\beqn
 \Vert s \Vert_Q^2= \frac{\im(\tau)^2}{4\pi^2} |\Delta(t)|^\frac{1}{6} \,.
\eeqn
\end{proposition}
\begin{proof}
For each $t \in J^* = J  - \{ 0,1, \infty\}$ we identify $\mathcal{E}_t = \mathbb{C} / \langle 1, \tau \rangle$ such that $j(\tau)=\mathsf{j}(t)=t$ and $\Delta_\tau=\Delta(t)$. We identify the one-form $dx/y$ in Lemma~\ref{lem:relat_canonical} with $dz$ in each smooth fiber $\mathcal{E}_t$ generating $\operatorname{coker}{\delbar_t}$. Similarly, the constant function $1$ generates $\ker{\delbar_t}$. We have $\Vert 1 \Vert^2_{L^2} = \Vert dz \Vert^2_{L^2} = 2\im(\tau)$, thus the canonical section $s$ of $\ker\delbar^* \otimes \operatorname{coker}\delbar$ satisfies $\Vert s \Vert_{L^2} =1$ and 
\beq
  \Vert s \Vert^2_Q = \sideset{}{'}\det \delbar_t^\dagger \delbar_t \;  \Vert  s \Vert^2_{L^2} = \frac{\im(\tau)^2}{4\pi^2} |\Delta_\tau|^\frac{1}{6} \,,
 \eeq 
where we used Equation~(\ref{eqn:at}) for the analytic torsion. 
\end{proof}
\subsection{The RRGQ formula}
\label{ssec:RRGQ}
We observe that since $\Delta(t)=0$ for $t=0,1,\infty$, we have that the Quillen norm vanishes at the punctures on $J^*$. The holomorphic determinant line bundle $\operatorname{\mathbf{Det}}{\delbar} \to J^*= \mathbb{C} P^1  - \{ 0,1, \infty\}$ is locally trivial by means of the section $s$ in Theorem~\ref{thm1} and in general does not extend to a bundle on the entire $j$-line $J\cong \mathbb{C} P^1$. Using the section $s$ we can compute the first Chern class  of the bundle $\operatorname{\mathbf{Det}}{\delbar} \to J^*$. We can then extend the curvature form representing the first Chern class to the entire $j$-line $J$ by allowing so-called \emph{currents}. These currents reflect the monodromy of $s$ around the punctures of $J^*$. 
\par In general,  let $(\xi,\Vert \cdot \Vert )$ be a holomorphic vector bundle over a complex algebraic variety, equipped with a smooth Hermitian metric. Then there exists a unique connection on $\xi$ compatible with the holomorphic structure and the metric $\Vert \cdot \Vert$. From Chern-Weil theory, there is a differential form representing the first Chern class $c_1(\xi,\Vert \cdot \Vert )$ in the de~\!Rham cohomology. We have the following result \cite{MR958789}:
\begin{lemma}
The representative in the de~\!Rham cohomology of the first Chern class $c_1(\xi,\Vert \cdot \Vert )$ is given by
\beqn
\label{eqn:CW}
 c_1\big(\xi,\Vert \cdot \Vert\big)=\frac{1}{2\pi i} \, \del\delbar \log \Vert s\Vert^2
\eeqn 
where $s$ is a nonzero holomorphic section of $\xi$. 
\end{lemma}
\par Trying to extend the bundle from $J^*$ to $J$, the points where $s$ vanishes can cause problems, since for at these points we may obtain \emph{current contributions} to the curvature, which are distribution-valued differential forms. Currents arise naturally from the classical Cauchy integral formula for single variable complex analysis \cite{MR1288523}, such as
\beqn
\label{eqn:current}
\delbar\left(\frac{1}{2\pi i} \frac{dt}{t}\right)=\delta_{(t=0)} 
\eeqn 
where $\delta_{(t=0)}$ is the Dirac delta function centered at $t=0$. The current contributions in a generalized first Chern class encode the holonomy of the trivializing section $s$ of the holomorphic determinant line bundle around the punctures of $J^*$. The computation of the generalized first Chern class is achieved by the so-called \emph{Riemann-Roch-Grothendieck-Quillen formula} or RRGQ formula for short; see \cites{MR830618, MR888370, BismutFreed, MR783704}. In the special situation of the Jacobian elliptic surface $\pi: S \to J \cong \mathbb{C} P^1$ given by Equation~(\ref{eqn:Ubdle}) and the holomorphic determinant line bundle $\operatorname{\mathbf{Det}}{\delbar} \to J^*= J  - \{ 0,1, \infty\}$ constructed above, we have the following:
\begin{theorem}[RRGQ]\label{thm:rrgq} 
In the situation described above, the generalized first Chern class of the determinant line bundle $\operatorname{\mathbf{Det}}{\delbar} \to J^*$ is given by
\beqn
\label{eqn:c1}
 c_1\big(\operatorname{\mathbf{Det}}{\delbar},\Vert \cdot \Vert_Q\big) = - \frac{1}{12}\big(2 \, \delta_{(t=0)}+9\, \delta_{(t=1)}+\delta_{(t=\infty)}\big)  + c_1(\mathcal{K}) 
 \eeqn
where $c_1(\mathcal{K})$ is the first Chern class of the line bundle $\mathcal{K} = \pi_* K_{S | J} \cong \mathcal{O}(1) \to J$.
 \end{theorem}
 \begin{proof}  
 Recall that the Quillen norm of the canonical holomorphic section $s$ in Proposition~\ref{thm1} is given by $\Vert s \Vert_Q^2 =\frac{\im(\tau)^2}{4\pi^2} |\Delta(t)|^\frac{1}{6}$. The discriminant $\Delta(t)$ of the Weierstrass equation defining the Jacobian elliptic surface $\pi: S \to J\cong \mathbb{C} P^1$ was given in  Equation~(\ref{eqn:Discriminant})  where we found $\Delta(t) = 27^3t^2(t-1)^9$ which vanishes at $t \in \{0,1,\infty\}$.  
\par Plugging the canonical section $s$ of Theorem~\ref{thm1} into Equation~(\ref{eqn:CW}) and applying the argument principle of Equation~(\ref{eqn:current}) yields
\beq
 \frac{1}{2\pi i}\del\delbar \log \Vert s \Vert_Q^2 = - \frac{1}{12}\big(2\delta_{(t=0)}+9\delta_{(t=1)}+\delta_{(t=\infty)}\big) 
 +j^*\left( \frac{i}{4\pi\im(\tau)^2} \, d\tau \wedge d\overbar{\tau} \right) \,,
\eeq
where we used the $j$-function $j: \mathbb{H}/\operatorname{PSL}(2,\mathbb{Z}) \to J$ with $j(\tau)=t$. Since the Poincar\'e metric on the hyperbolic upper half plane is given by
\beq
 \frac{i}{2\pi \im(\tau)^2}d\tau \wedge d\overbar{\tau} =  \frac{dx \wedge dy }{\pi y^2} \,,
\eeq
and the $j$-function is a complex diffeomorphism, it follows that
 \beq
  j^*\left( \frac{i}{2\pi\im(\tau)^2} \, d\tau \wedge d\overbar{\tau} \right) = c_1\left( T_\mathbb{C} J^{(1,0)} \right) \,.
\eeq
We have $J\cong \mathbb{C} P^1$ and $T_\mathbb{C} J^{(1,0)} \cong \mathcal{O}(2)$. Since the first Chern class of products of line bundles on $J$ is additive, i.e., $c_1(\xi_1 \otimes \xi_2) = c_1(\xi_1)+c_1(\xi_2)$ it follows that the continuous part of $c_1(\operatorname{\mathbf{Det}}{\delbar})$ is the first Chern class of $\mathcal{O}(1)$. Using Lemma~\ref{lem:relat_canonical} the claim follows.
\end{proof}
\subsection{Extension as a meromorphic connection}
The RRGQ formula is the key to studying the anomalies of the fiberwise Cauchy-Riemann operator $\delbar_t$ on a Jacobian elliptic surface. In particular, in the smooth category, the \emph{local anomaly} is the first Chern class of the determinant line bundle \cite{MR915611}, and it is computed as the line bundle's curvature tensor. Hence if the curvature vanishes, then the bundle has no \emph{local} anomaly. Conversely, Theorem \ref{thm:rrgq} proves that the determinant line bundle of the fiberwise Cauchy-Riemann operators $\delbar_t$ on the Jacobian elliptic surface $\pi: S \to J\cong \mathbb{C} P^1$ given by Equation~(\ref{eqn:Ubdle}) has a non-vanishing local anomaly.
\par Moreover, the current contributions encode the holonomy of the trivializing section $s$ for the bundle $\operatorname{\mathbf{Det}}{\delbar} \to J^*=J - \{  0, 1, \infty \}$ around the punctures over which the Weierstrass model has singular fibers. This represents the \emph{global} anomaly of the bundle. To analyze the holonomy group we have the following:
\begin{lemma}
\label{lem:split1}
There is a flat holomorphic line bundle $\mathcal{M}^* \to J^*$ with $\mathbb{Z}_{12}$-holonomy such that $\operatorname{\mathbf{Det}}{\delbar} \cong \mathcal{K} \otimes \mathcal{M}^*$.
\end{lemma}
\begin{proof} 
Since the first Chern class of products of line bundles on $J^*$ is additive, i.e., $c_1(\xi_1 \otimes \xi_2) = c_1(\xi_1)+c_1(\xi_2)$, it follows from Equation~(\ref{eqn:c1}) that $\operatorname{\mathbf{Det}}{\delbar} \cong \mathcal{K} \otimes \mathcal{M}^*$  where $\mathcal{K}$ is (the restriction of) the bundle $\mathcal{K} \cong \mathcal{O}(1) \to J \cong \mathbb{C} P^1$. The discriminant $\Delta(t) = 27^3t^2(t-1)^9$ vanishes at $t_0 \in \{0,1,\infty\}$.  At each point $t_0$, we compute the holonomy of $\Delta(t)^{1/12}$ (after fixing a base point of a smooth fiber) by encircling the point $t_0$ via the path $t = t_0+ \frac{1}{2} e^{2\pi i \epsilon}$ and calculating the result as $\epsilon \to 1$. At $t_0=0$, we obtain $\Delta^{1/12}(t_0) \mapsto e^{\pi i /3}\Delta^{1/12}(t_0)$. Similarly,  at $t_0=1$, we obtain $\Delta^{1/12}(t_0) \mapsto e^{\pi i /2}\Delta^{1/12}(t_0)$,  and at $t_0=\infty$, we obtain $\Delta^{1/12}(t_0) \mapsto e^{\pi i /6}\Delta^{1/12}(t_0)$. The smallest subgroup of $\operatorname{U}(1)$ that contains these generators is $\mathbb{Z}_{12}$. Furthermore, this holonomy is not present in the holomorphic cotangent bundle on $\mathbb{C} P^1$, we conclude that the desired flat line bundle $\mathcal{M}^*$ with $\mathbb{Z}_{12}$-holonomy exists as claimed.
\end{proof}
\par We offer the following interpretation of the anomalies: when $t \neq 0,1,\infty$, the automorphism group of the elliptic curve $\mathcal{E}_t$ is isomorphic to $\mathbb{Z}_2$. When $t=0,1,\infty$, the automorphism group is isomorphic to $\mathbb{Z}_4, \mathbb{Z}_6$, and $\mathbb{Z}_4$, respectively \cite{MR0344216}. These points give rise to the global anomaly. The jumping behavior in the symmetry of the elliptic fiber is also called the \emph{holomorphic anomaly}; see \cite{MR915611}. 
\par The flat holomorphic line bundle $\mathcal{M}^* \to J^*$ can be extended to a line bundle $\mathcal{M} \to J$ with a flat meromorphic connection $\delbar_t$ that has only regular singular points. This follows from the \emph{Riemann-Hilbert correspondence} which asserts that the restriction to $J^*$ is an equivalence of categories between the category of flat meromorphic connections on $J$ with only regular singular points and holomorphic on $J^*$ and the category of flat holomorphic connections on $J^*$ \cite{MR0417174}. Here, it simply means that there exists a trivialization on a flat line bundle $\mathcal{M} \to J$ so that, when restricted to a punctured disc $D^*_{t_0}$ around any point $t_0 \in J - J^*$, the $\delbar_t$-operator on $\mathcal{M}$ is given by
\beqn
 \delbar_t \Big\vert_{D^*_{t_0}} = \delbar - \frac{a \, dt}{t-t_0} + \eta \,,
\eeqn 
where $a \in \mathbb{C}$ and $\eta \in \Omega^1$ is a holomorphic one-form on $D_{t_0}$. We have the following:
\begin{proposition}
\label{prop:split}
The flat holomorphic line bundle $\mathcal{M}^* \to J^*$ extends to a line bundle $\mathcal{M} \to J$ with a flat meromorphic connection and regular singular points over $J - J^*$.
\end{proposition}
\begin{proof}
Let $(\mathcal{M}^*, \delbar)$ be the holomorphic flat connection on $J^*$. When restricted to a punctured disc $D^*_{t_0}$ around a point $t_0 \in J - J^*$, it is therefore determined by some monodromy matrix $A \in \operatorname{U}(1)$. Taking logarithms, there exists $a \in \mathfrak{u}(1)$ such that $A=\exp{(2\pi i a)}$. Then, the meromorphic connection on $D_{t_0}$ given by $\delbar - \frac{a \, dt}{t-t_0}$ has flat sections of the form $t \mapsto v \, \exp{(a \log{(t-t_0)})}$ for any $v \in \mathbb{C}^*$. These sections have monodromy around $t_0$ given by $A=\exp{(2\pi i a)}$, so have their restriction to $D^*_{t_0}$.
\end{proof}
\section{Twisting and anomaly cancellation}
\label{section:gravanom}
The Riemann-Roch-Grothendieck-Quillen (RRGQ) formula has a twisted analogue, similar to the twisted version of the signature theorem in Equation~(\ref{eqn:AStwist}). However, to state this formula we will replace the Jacobian elliptic surface $S \to J$ given by Equation~(\ref{eqn:Ubdle}) with a Jacobian elliptic surface $\pi: Z \to B$ whose Weierstrass model has only nodes, i.e., fibers of Kodaira-type $I_1$; we will explain below how such a surface can be constructed using Corollary~\ref{cor:Jmap}. This setup has the advantage that the total space of Equation~(\ref{eqn:Ubdle}) is smooth, and no additional blowups are needed to move from its total space to $Z$. 
\par Let $\xi \to Z$ be a \emph{holomorphic} vector bundle of rank $r$ with a smooth Hermitian metric. Then there is a unique unitary connection on $\xi$ compatible with its holomorphic structure \cite{MR1288304}. Using this connection we can compute the Chern classes $c_i(\xi)$ for $i=1, 2$. We also obtain a \emph{twisted} Cauchy-Riemann operator $\overbar{\del}^\xi_t$ on any smooth elliptic curve $\mathcal{E}_t$ in the fibration $\pi: Z \to B$ coupled to the restriction of the holomorphic bundle $\xi$ given by
\beqn
\label{eqn:twisted_del}
 \delbar^\xi_t: C^\infty\big( \mathcal{E}_t, \xi\Big\vert_{\mathcal{E}_t}\big) \to C^\infty\big( \mathcal{E}_t , \overbar{K}_{\mathcal{E}_t} \otimes \xi\big\vert_{\mathcal{E}_t}\big)  \,.
\eeqn 
For the family of twisted operators $\{ \delbar^{\, \xi}_t \}_{t \in B}$ a determinant line bundle $\operatorname{\mathbf{Det}}{\delbar^{\, \xi}} \to B$ together with a Quillen metric $\Vert \cdot \Vert_Q$ can be constructed as before; see \cites{MR769356, MR915611}. The analogue of Theorem~\ref{thm:rrgq} is the following:
 \begin{theorem}[Twisted RRGQ]
\label{thm:trrgq} 
In the situation described above, the generalized first Chern class of the determinant line bundle $\operatorname{\mathbf{Det}}{\delbar^{\, \xi}}$ is given by
\beqn
\label{eqn:tQQGQ}
 c_1\big(\operatorname{\mathbf{Det}}{\delbar^{\, \xi}},\Vert \cdot \Vert_Q\big) = r \cdot c_1\big(\operatorname{\mathbf{Det}}{\delbar},\Vert \cdot \Vert_Q\big)  -  \int_{X |B} c_2(\xi) \,,
 \eeqn
where $c_2(\xi)$ is the second Chern class of the holomorphic vector bundle $\xi \to Z$ of rank $r$ assumed to satisfy $c_1(\xi)=0$. Here, $\int_{X |B}  c_2(\xi)$ is understood as integrating the four-form $c_2(\xi)$ over the vertical fibers of $\pi: Z \to B$. 
 \end{theorem}
\par Notice that Equation~(\ref{eqn:tQQGQ}) is the generalization of Equation~(\ref{eqn:AStwist}) for families. However, the last term on the right hand side of Equation~(\ref{eqn:tQQGQ}) now yields upon integration a two-form on the base of the fibration. As we will show, Theorem~\ref{thm:trrgq} then implies that we can always choose the holomorphic vector bundle $\xi$ in such a way that the local anomaly of the determinant line bundle is \emph{canceled}.
\subsection{The generic elliptic surface}
As an application of Corollary~\ref{cor:Jmap}, we will consider the case of the most generic rational Jacobian elliptic surface $Z \to B\cong \mathbb{C} P^1$. This is, a Jacobian elliptic fibration whose only singular fibers are twelve nodes, i.e., fibers of Kodaira-type $I_1$. This means that the discriminant $\Delta(t)$ has 12 distinct simple roots, and we have
\beqn
\label{eqn:Discr1}
 \Delta(t) = \prod_{i=1}^{12} (t-t_i) \,,
\eeqn
for the distinct points $t_1,\dots, t_{12} \in B$, and we set $B^* = B - \{ t_1, \dots, t_{12} \}$. It was shown in \cite{MR1104782} that this Jacobian elliptic surface exists; it was denoted by \#1 in the complete classification of rational Jacobian elliptic surfaces in \cite{MR1104782}. It follows from general arguments in \cites{MR0165541}  that the total space of Equation~(\ref{Weierstrass}) is rational and smooth. Similarly, there is the Jacobian elliptic surface where $g_2$ and $g_3$ are generic  polynomials of degree $8$ and $12$, respectively, and the total space of Equation~(\ref{Weierstrass}) is a smooth $K3$ surface. In this case, the singular fibers are 24 nodes, i.e., fibers of Kodaira-type $I_1$, and one has a discriminant $\Delta(t)$ with 24 distinct simple roots, i.e.,
\beqn 
\label{eqn:Discr2}
 \Delta(t) = \prod_{i=1}^{24} (t-t_i) \,,
\eeqn
for the distinct points $t_1,\dots, t_{24} \in B$, and we set $B^* = B - \{ t_1, \dots, t_{24} \}$. We can adopt the construction of the holomorphic determinant line bundle from Section~\ref{ssec:RRGQ} to obtain $\operatorname{\mathbf{Det}}{\delbar} \to B^*$ in both cases. The only difference between the case $n=1$ (rational surface) and $n=2$ ($K3$ surface) is that the degree of the holomorphic map $\mathsf{j}(t)=j(\tau)$ is $1$ or $2$, respectively. We adopt the proofs of Theorem~\ref{thm:rrgq}, Lemma~\ref{lem:split1}, Lemma~\ref{lem:relat_canonical}, and Proposition~\ref{prop:split} to obtain the following:
\begin{corollary} 
\label{cor:RRGQ_Z}
Let $Z \to B\cong \mathbb{C} P^1$ be the Jacobian elliptic surface whose Weierstrass model has $12n$ singular fibers of Kodaira-type $I_1$ for $n=1$ (rational surface) or $n=2$ (K3 surface). The generalized first Chern class of the determinant line bundle $\operatorname{\mathbf{Det}}{\delbar} \to B^*$ is given by
\beqn
\label{eqn:c1_Z}
 c_1\big(\operatorname{\mathbf{Det}}{\delbar},\Vert \cdot \Vert_Q\big) = -  \frac{1}{12}\left(\sum_{i=1}^{12n}\delta_{(t=t_i)}\right) + c_1(\mathcal{K}) 
 \eeqn
with $\mathcal{K} = \pi_* K_{Z | B}\cong \mathcal{O}(n) \to B$. Moreover, there is a flat holomorphic line bundle $\mathcal{M}^* \to B^*$ with a $\mathbb{Z}_{12}$-holonomy such that $\operatorname{\mathbf{Det}}{\delbar} \cong \mathcal{K} \otimes \mathcal{M}^*$. In turn, the flat holomorphic line bundle $\mathcal{M}^* \to B^*$ extends to a line bundle $\mathcal{M} \to B$ with a flat meromorphic connection $\delbar_t$ for $t \in B^*$ given by
\beqn
\label{eqn:connection}
 \delbar_t  = \delbar - \frac{1}{12} \sum_{i=1}^{12n} \frac{dt}{t-t_i} \,.
\eeqn
 \end{corollary}
\par Roughly speaking, pulling back the curvature from $\operatorname{\mathbf{Det}}{\delbar} \to J^*$ has the effect of spreading out the current contributions over different points on the elliptic fibration, while the total flux of the current contributions is fixed by $\deg{\Delta}=12n$.  In fact, the canonical section $s$ in Theorem~\ref{thm1} has  holonomy given by $\Delta^{1/12} \mapsto e^{\pi i /6}\Delta^{1/12}$. Hence, the nontrivial holonomy group of $\mathcal{M}^*$ is $\mathbb{Z}_{12} \subseteq \operatorname{U}(1)$.  We also make the following:
\begin{remark}
It follows from results in \cites{MR1027535,MR2815730} that for a Jacobian elliptic surface $Z \to B\cong \mathbb{C} P^1$ with only nodes in its Weierstrass model a suitable notion of a $\delbar$-operator and its regularized determinant can be established for all fibers, including the nodes, so that the meromorphic connection in Equation~(\ref{eqn:connection}) arises as a meromorphic connection of the extended determinant line bundle over $B$.
\end{remark}
The invariant computed in Equation~(\ref{eqn:c1_Z}) from the family of fiberwise signature operators is a refined invariant of the elliptic fibration $Z \to B$, compared to the index of the signature operator on the total space $Z$ in Equation~(\ref{eqn:AS}). General topological arguments show that the latter is simply $-8n$ for $n=1$ (rational surface) or $n=2$ (K3 surface). 
\subsection{The Poincar\'e line bundle}
Let $\pi: Z \to B \cong \mathbb{C} P^1$ be the Jacobian elliptic surface with a zero-section denoted by $\sigma: B \to Z$ and a Weierstrass model with $12n$ singular fibers of Kodaira-type $I_1$ for $n=1, 2$. Then, $Z$ is the total space of Equation~(\ref{Weierstrass}), is smooth, and a rational surface for $n=1$ and a $K3$ surface for $n=2$. This is important because it means that we can simply ignore all singularities when constructing the fiber product of $Z$. We also assume that the group of sections for the Jacobian elliptic surface $Z$ admits no two-torsion. This will allow the restriction of a holomorphic $\operatorname{SU}(2)$ bundle $\xi \to Z$ of rank two to every fiber $\mathcal{E}_t = \pi^{-1}(t)$ to be an extension bundle. 
\par We first build a rank-two $\operatorname{SU}(2)$ bundle over a smooth elliptic curve $\mathcal{E}$. It follows from results in \cite{MR0131423} that a rank-two vector bundle $V \to \mathcal{E}$ is a (semi-stable) holomorphic $\operatorname{SU}(2)$ bundle if and only if $V \cong \mathcal{N}_1 \oplus \mathcal{N}_2$ for two holomorphic line bundles $\mathcal{N}_1,\mathcal{N}_2 \to \mathcal{E}$ with $\mathcal{N}_1 \otimes \mathcal{N}_2 \cong \mathcal{O}$.  For simplicity, we assume that the line bundles $\mathcal{N}_1, \mathcal{N}_2$ have degree zero. Then, there are unique points $q_1, q_2 \in \mathcal{E}$ such that $\mathcal{N}_1, \mathcal{N}_2$ each have a holomorphic section vanishing at a point $q_1$ and $q_2$ respectively, and a simple pole at $p=\infty$, i.e., the neutral point of the elliptic group law. Using the group law on $\mathcal{E}$, the condition $\mathcal{N}_1 \otimes \mathcal{N}_2 = \mathcal{O}$ implies $q_1+q_2=0$. Hence, we write $q_1=q$, $q_2=-q$, and $V=\mathcal{O}(q-p)\oplus \mathcal{O}(-q-p)$.  For each such a pair $(q,-q) \in \mathcal{E}\times \mathcal{E}$, there is a meromorphic function $w=a_0-a_2x$ on $\mathcal{E}$ given by Equation~(\ref{Weierstrass}) with $a_0, a_2 \in \mathbb{C}$ that vanishes at $q$ and $-q$ and has a simple pole at $p$. That is, we think of  the points $\pm q \in \mathcal{E}$ as given by the coordinates 
 \beq 
  x=\frac{a_2}{a_0}\,, \quad y=\pm \sqrt{4\left(\frac{a_2}{a_0}\right)^3-g_2 \, \frac{a_2}{a_0}-g_3}\,.
\eeq 
\par We also introduce the \emph{Poincar\'e line bundle}: for a smooth elliptic curve $\mathcal{E}$, the degree-zero holomorphic line bundles over a smooth elliptic curve $\mathcal{E}$ are parameterized by $\mathcal{E}$ itself since each point $q \in \mathcal{E}$ corresponds to the line bundle $\mathcal{O}(q-p)$. We denote by $\Delta$ the diagonal in $\mathcal{E} \times \mathcal{E}$. The Poincar\'e line bundle $\mathcal{P} \to \mathcal{E} \times \mathcal{E}$ is obtained from the divisor 
\beq 
 D = \Delta - \mathcal{E} \times \{p\} - \{p\} \times \mathcal{E}, 
\eeq
by setting $\mathcal{P}=\mathcal{O}_{\mathcal{E} \times \mathcal{E}}(D)$ so that  $\mathcal{P}\vert_{\{q\}\times \mathcal{E}} \cong \mathcal{P}\vert_{\mathcal{E} \times \{q\}} \cong \mathcal{O}(q-p)$. 
\par Now let the elliptic curve $\mathcal{E}$ vary over the elliptic fibers $\mathcal{E}_t$ of the Jacobian elliptic surface $Z \to B \cong \mathbb{C} P^1$ with $t \in B$ such that the point at infinity in each fiber is given by the zero section $p=\sigma(t)$. Next, we consider a pair of points $\pm q$ which are the solutions of $w = a_0 - a_2x=0$ where the coefficients $a_i$ are sections $a_i \in \Gamma(B, \mathcal{R} \otimes \mathcal{L}^{-i})$ for a non-trivial holomorphic line bundle  $\mathcal{R} \to B$ and the normal bundle $\mathcal{L} \to B$ introduced at the end of Section~\ref{section:ellipticsurfaces}. In this way, the vanishing locus of the global section $w \in \Gamma(B, \mathcal{R})$ defines a ramified double covering $C_\mathcal{R} \subset Z$ of $B$, called a \emph{spectral double cover}. 
\par From the total space $Z$ we form the \emph{fiber product}, given by 
 \beq 
  Z \times_B Z = \{(z_1,z_2) \in Z \times Z \, | \, \pi(z_1) = \pi(z_2)\},
\eeq 
with a holomorphic projection map $\widetilde{\pi} : Z \times_B Z \to B$ given by $\widetilde{\pi}(z_1,z_2)=\pi(z_1)$; this is well defined by virtue of the definition of $Z \times_B Z$ and $\pi(z_1)=\pi(z_2)$. For $t \in B$, we have $\widetilde{\pi}^{-1}(t)=\mathcal{E}_t \times \mathcal{E}_t$ with $\mathcal{E}_t = \pi^{-1}(t)$. From the spectral cover $C_\mathcal{R} \subset Z$ we obtain, by using the fiber product, the topological subspace $C_\mathcal{R} \times_B Z \subset Z \times_B Z$ with $z_1 \in C$. The map $\pi_2 : C_\mathcal{R} \times_B Z \to Z$ obtained by forgetting $z_1$ is a two-fold covering. 
\par The equation $z_1=z_2$ forms a divisor $\Delta \subset Z \times_B Z$. The Poincar\'e line bundle $\mathcal{P} \to Z \times_B Z$ on the Jacobian elliptic surface $\pi: Z \to B$ with section $\sigma$ is obtained from the divisor 
\beq 
 D= \Delta - Z \times \sigma - \sigma \times Z \,, 
\eeq
by setting $\mathcal{P}=\mathcal{O}(D) \otimes \widetilde{\pi}^*\mathcal{L}$ where $\mathcal{L} \to B$ is the aforementioned normal bundle of the Jacobian elliptic fibration $\pi: Z \to B$. By restriction, we obtain the restricted Poincar\'e line bundle $\mathcal{P}_\mathcal{R} \to C \times_B Z$. Using results of \cite{MR1468319}, we have the following:
\begin{proposition}
\label{prop:bundle}
Given a spectral double cover $C_\mathcal{R} \subset Z \to B$ corresponding to an ample line bundle $\mathcal{R} \to B$, the bundle
\beqn
\label{eqn:vecbundle}
 \xi = \pi_{2*} \big( \mathcal{P}_\mathcal{R}\big) \to Z
\eeqn
is a rank-two holomorphic $\operatorname{SU}(2)$ bundle over $Z$. 
\end{proposition}
\begin{proof}
For any $z \in Z$ which is not in the branching locus of $\pi_2$ with $t=\pi(z) \in B$, we have $C_{\mathcal{R}, t} = \{ y_1, y_2\}$ and
\beq 
 \xi_z = \mathcal{P}_{(y_1,z)} \oplus \mathcal{P}_{(y_2,z)}\,.
\eeq
Thus, the restriction of $\xi$ to $\mathcal{E}_t = \pi^{-1}(t)$ is a sum of degree-zero line bundles given by
\beq
 \xi \Big\vert_{\mathcal{E}_t} = \mathcal{O} \big( -q(t) + \sigma(t) \big) \oplus \mathcal{O} \big( q(t) + \sigma(t) \big) \,,
\eeq
where $\pm q(t)$ are obtained as the solutions with $x$-coordinate given by  $w = a_0(t) - a_2(t) \,x=0$. The restriction of $\xi$ to any such fiber $\mathcal{E}_t = \pi^{-1}(t)$ carries a flat $\operatorname{SU}(2)$ connection. At the branching points of $\pi_2$ the preimage of $t \in B$ is a point of multiplicity two. Thus, the restriction $\xi \vert_{\mathcal{E}_t}$  is a non-trivial extension of a line bundle by a second isomorphic line bundle.  This restriction bundle admits no flat $\operatorname{SU}(2)$ connection. To fit these two types of bundles together to form a holomorphic bundle on $Z$ we replace some of the flat bundles by non-isomorphic, S-equivalent bundles. It follows from the results in~\cite{MR1288304} that after fitting these bundles together, we obtain a bundle with a Hermitian $\operatorname{SU}(2)$ connection.
\end{proof}
\par The following is a crucial computation in \cite{MR1468319} which we cite without proof:
\begin{lemma}\label{lemma:su2bundle} 
In the situation of Proposition~\ref{prop:bundle} we have $\pi_{*} c_2(\xi) = c_1(\mathcal{R})$. 
\end{lemma}
\subsection{Canceling the local anomaly} 
We now prove our main theorem:
\begin{theorem} 
\label{thm:main}
Let $Z \to B\cong \mathbb{C} P^1$ be the Jacobian elliptic surface whose Weierstrass model has $12n$ singular fibers of Kodaira-type $I_1$ over $\{t_i\}_{i=1}^{12n}$ for $n=1$ (rational surface) or $n=2$ ($K3$ surface). Let $C_\mathcal{R} \subset Z \to B$ be the spectral double cover corresponding to the line bundle $\mathcal{R}= \mathcal{O}(2n) \to B$ that yields the rank-two holomorphic $\operatorname{SU}(2)$ bundle  $\xi = \pi_{2*} \big( \mathcal{P}_\mathcal{R}\big) \to Z$. Then, the generalized first Chern class of the determinant line bundle $\operatorname{\mathbf{Det}}{\delbar^{\, \xi}}\to B^*$ is given by
\beqn
\label{eqn:tQQGQ_special}
 c_1\big(\operatorname{\mathbf{Det}}{\delbar^{\, \xi}},\Vert \cdot \Vert_Q\big) = - \frac{1}{6}\left(\sum_{i=1}^{12n}\delta_{(t=t_i)}\right)\,.
\eeqn
In particular, there is no local anomaly.
\end{theorem}
\begin{proof} 
In Theorem~\ref{thm:trrgq} we use the rank-two $(r=2)$ bundle $\xi \to Z$ constructed in Proposition~\ref{prop:bundle} with the contribution coming from the twist
computed in Lemma~\ref{lemma:su2bundle} and Corollary~\ref{cor:RRGQ_Z}. The continuous part of the  first Chern class of the determinant line bundle $\operatorname{\mathbf{Det}}{\delbar^{\, \xi}}\to B^*$ is given by
\beq
  r\cdot c_1(\mathcal{K}) -  \pi_* c_2(\xi) = r\cdot c_1(\mathcal{K}) - c_1(\mathcal{R}) \,.
\eeq  
It follows from Corollary~\ref{cor:RRGQ_Z} that for $\mathcal{R} \cong \mathcal{O}(2n)$ the continuous part of the first Chern class vanishes. Notice that a global section $w \in H^0(B, \mathcal{R})$ (defining the rank-two holomorphic $\operatorname{SU}(2)$ bundle  $\xi \to Z$) exists because the bundle $\mathcal{O}(2)$ is already very ample and defines an embedding $\mathbb{C} P^1 \hookrightarrow \mathbb{C} P^2$ by $[z_0 : z_1] \mapsto [z_0^2 : z_0 z_1 : z_1^2]$.
\end{proof}
\par The global anomaly, represented by the current contributions in Equation~(\ref{eqn:tQQGQ_special}), is of critical importance in string theory \cite{MR1134801}. In fact, the global anomaly sheds light on why certain extended objects, called \emph{D-branes}, short for Dirichlet membrane, have to be inserted when constructing \emph{string compactifications}. The traditional approach to producing low-dimensional physical models out of high-dimensional theories such as the string theories and M-theory has been to use a specific geometric compactification of the ``extra'' dimensions and derive an effective description of the lower-dimensional theory from the choice of geometric compactification.  However, it has long been recognized that there are other possibilities:  for example, one can couple perturbative string theory to an arbitrary superconformal
two-dimensional theory (geometric or not) to obtain an effective perturbative string compactification in lower dimensions. One way of making an analogous construction in non-perturbative string theory is to exploit the nonperturbative duality transformations which relate various compactified string theories (and M-theory) to each other.  This idea was the basis of the construction of F-theory \cite{MR1403744}.
\par In a standard compactification of the type IIB string, $\tau$ is a constant and D7-branes (D-branes are typically classified by their spatial dimension) are absent.  Vafa's idea in proposing F-theory \cite{MR1403744} was to simultaneously allow a variable $\tau$ and the D7-brane sources, arriving at a new
class of models in which the string coupling is never weak. Thus, one of the fundamental interpretations of F-theory is in terms of the type IIB string, where it depends on three ingredients:  an  $\operatorname{PSL}(2, \mathbb{Z})$ symmetry of the theory, a complex scalar field $\tau$ (the axio-dilaton) with positive imaginary part (in
an appropriate normalization)  on  which $\operatorname{PSL}(2, \mathbb{Z})$ acts by fractional linear  transformations, and D7-branes, which serve as a source for the
multi-valuedness of $\tau$ if $\tau$ is allowed to vary. To do so, one needs to know what types of seven-branes have to be inserted. It turns out that there is a complete dictionary between the different types of seven-branes which must be inserted and the possible singular limits in one-parameter families
of elliptic curves given by the work of Kodaira  \cite{MR0184257} and N\'eron \cite{MR0179172}. Thus, any Jacobian elliptic surfaces $\pi: Z \to B\cong \mathbb{C} P^1$ is a good candidate for such an F-theory background in an eight-dimensional string compactifications. However, because of supersymmetry considerations only the case $n=2$ in Theorem~\ref{thm:main}, that is $Z$ is a $K3$ surface, turns out to be viable.
\par For a trivial family of elliptic curves, the family of $\delbar$ operators has no current contributions and no global anomaly. In contrast, the total space of the Jacobian elliptic surface $\pi: Z \to B\cong \mathbb{C} P^1$ in the case $n=2$ ($K3$ surface) in Theorem~\ref{thm:main} can now be interpreted as an F-theory background with a variable $\tau$ and 24 disjoint D7-branes inserted into the physical theory. In the context of the physical description of the corresponding compactification of the type IIB string theory, Theorem~\ref{thm:main} provides an explanation why and where these D7-branes have to be inserted: they have to be inserted at the points where $\Delta(t)=0$ in order to cancel the current contributions of the generalized first Chern class of the determinant line bundle which plays a key role in the description of the path integral description of the effective physical theory \cite{MR1797580}.
\bibliographystyle{alpha}

\begin{thebibliography}{99}
\bib{MR769356}{article}{
   author={Alvarez, Orlando},
   author={Singer, I. M.},
   author={Zumino, Bruno},
   title={Gravitational anomalies and the family's index theorem},
   journal={Comm. Math. Phys.},
   volume={96},
   date={1984},
   number={3},
   pages={409--417},
   issn={0010-3616},
   review={\MR{769356}},
   }

\bib{MR0131423}{article}{
   author={Atiyah, M. F.},
   title={Vector bundles over an elliptic curve},
   journal={Proc. London Math. Soc. (3)},
   volume={7},
   date={1957},
   pages={414--452},
   issn={0024-6115},
   review={\MR{0131423}},
   doi={10.1112/plms/s3-7.1.414},
   }

\bib{MR157392}{article}{
   author={Atiyah, M. F.},
   author={Singer, I. M.},
   title={The index of elliptic operators on compact manifolds},
   journal={Bull. Amer. Math. Soc.},
   volume={69},
   date={1963},
   pages={422--433},
   issn={0002-9904},
   review={\MR{157392}},
   doi={10.1090/S0002-9904-1963-10957-X},
}

\bib{MR0279833}{article}{
   author={Atiyah, M. F.},
   author={Singer, I. M.},
   title={The index of elliptic operators. IV},
   journal={Ann. of Math. (2)},
   volume={93},
   date={1971},
   pages={119--138},
   issn={0003-486X},
   review={\MR{0279833}},
   doi={10.2307/1970756},
}

\bib{MR830618}{article}{
   author={Belavin, A. A.},
   author={Knizhnik, V. G.},
   title={Algebraic geometry and the geometry of quantum strings},
   journal={Phys. Lett. B},
   volume={168},
   date={1986},
   number={3},
   pages={201--206},
   issn={0370-2693},
   review={\MR{830618}},
   doi={10.1016/0370-2693(86)90963-9},
}

\bib{MR888370}{article}{
   author={Belavin, A. A.},
   author={Knizhnik, V. G.},
   title={Complex geometry and the theory of quantum strings},
   language={Russian},
   journal={Zh. \`Eksper. Teoret. Fiz.},
   volume={91},
   date={1986},
   number={2},
   pages={364--390},
   translation={
      journal={Soviet Phys. JETP},
      volume={64},
      date={1986},
      number={2},
      pages={214--228 (1987)},
      issn={0038-5646},
   },
   review={\MR{888370}},
}

\bib{MR958789}{article}{
   author={Bismut, Jean-Michel},
   author={Bost, Jean-Beno\^{\i}t},
   title={Fibr\'{e}s d\'{e}terminants, m\'{e}triques de Quillen et d\'{e}g\'{e}n\'{e}rescence des courbes},
   language={French, with English summary},
   journal={C. R. Acad. Sci. Paris S\'{e}r. I Math.},
   volume={307},
   date={1988},
   number={7},
   pages={317--320},
   issn={0249-6291},
   review={\MR{958789}},
   doi={10.1007/BF02391902},
}

\bib{BismutFreed}{article}{
  author={Bismut, Jean-Michel},
  author={Freed, Daniel S.},
  title={Geometry of elliptic families I, II},
  journal={Orsay preprint}, 
  number={85T47},
  }

\bib{MR915611}{article}{
   author={Bost, J.-B.},
   title={Conformal and holomorphic anomalies on Riemann surfaces and
   determinant line bundles},
   conference={,
      title={VIIIth international congress on mathematical physics},
      address={Marseille},
      date={1986},
   },
   book={
      publisher={World Sci. Publishing, Singapore},
   },
   date={1987},
   pages={768--775},
   review={\MR{915611}},
}

\bib{MR0015793}{article}{
   author={Chern, Shiing-Shen},
   title={Characteristic classes of Hermitian manifolds},
   journal={Ann. of Math. (2)},
   volume={47},
   date={1946},
   pages={85--121},
   issn={0003-486X},
   review={\MR{0015793}},
   doi={10.2307/1969037},
}

\bib{MR0045432}{article}{
   author={Chern, Shiing-Shen},
   title={Differential geometry of fiber bundles},
   conference={
      title={Proceedings of the International Congress of Mathematicians,
      Cambridge, Mass., 1950, vol. 2},
   },
   book={
      publisher={Amer. Math. Soc., Providence, R. I.},
   },
   date={1952},
   pages={397--411},
   review={\MR{0045432}},
}

\bib{MR0011027}{article}{
   author={Chern, Shiing-Shen},
   title={A simple intrinsic proof of the Gauss-Bonnet formula for closed
   Riemannian manifolds},
   journal={Ann. of Math. (2)},
   volume={45},
   date={1944},
   pages={747--752},
   issn={0003-486X},
   review={\MR{0011027}},
   doi={10.2307/1969302},
}

\bib{MR0417174}{book}{
   author={Deligne, Pierre},
   title={\'{E}quations diff\'{e}rentielles \`a points singuliers r\'{e}guliers},
   language={French},
   series={Lecture Notes in Mathematics, Vol. 163},
   publisher={Springer-Verlag, Berlin-New York},
   date={1970},
   pages={iii+133},
   review={\MR{0417174}},
}

\bib{MR598586}{article}{
   author={Eguchi, Tohru},
   author={Gilkey, Peter B.},
   author={Hanson, Andrew J.},
   title={Gravitation, gauge theories and differential geometry},
   journal={Phys. Rep.},
   volume={66},
   date={1980},
   number={6},
   pages={213--393},
   issn={0370-1573},
   review={\MR{598586}},
   doi={10.1016/0370-1573(80)90130-1},
}

\bib{MR1288304}{book}{
   author={Friedman, Robert},
   author={Morgan, John W.},
   title={Smooth four-manifolds and complex surfaces},
   series={Ergebnisse der Mathematik und ihrer Grenzgebiete (3) [Results in
   Mathematics and Related Areas (3)]},
   volume={27},
   publisher={Springer-Verlag, Berlin},
   date={1994},
   pages={x+520},
   isbn={3-540-57058-6},
   review={\MR{1288304}},
   doi={10.1007/978-3-662-03028-8},
}

\bib{MR1468319}{article}{
   author={Friedman, Robert},
   author={Morgan, John},
   author={Witten, Edward},
   title={Vector bundles and ${\rm F}$ theory},
   journal={Comm. Math. Phys.},
   volume={187},
   date={1997},
   number={3},
   pages={679--743},
   issn={0010-3616},
   review={\MR{1468319}},
   doi={10.1007/s002200050154},
}

\bib{MR1554993}{article}{
   author={Fredholm, Ivar},
   title={Sur une classe d'\'{e}quations fonctionnelles},
   language={French},
   journal={Acta Math.},
   volume={27},
   date={1903},
   number={1},
   pages={365--390},
   issn={0001-5962},
   review={\MR{1554993}},
   doi={10.1007/BF02421317},
}
\bib{MR1134801}{article}{
   author={Freed, Dan},
   title={Anomalies and determinant line bundles},
   conference={
      title={XVIIth International Colloquium on Group Theoretical Methods in
      Physics},
      address={Sainte-Ad\`ele, PQ},
      date={1988},
   },
   book={
      publisher={World Sci. Publ., Teaneck, NJ},
   },
   date={1989},
   pages={186},
   review={\MR{1134801}},
}
\bib{MR1797580}{article}{
   author={Freed, Daniel S.},
   author={Witten, Edward},
   title={Anomalies in string theory with D-branes},
   journal={Asian J. Math.},
   volume={3},
   date={1999},
   number={4},
   pages={819--851},
   issn={1093-6106},
   review={\MR{1797580}},
   doi={10.4310/AJM.1999.v3.n4.a6},
}
\bib{Fubini}{article}{
   author={Fubini, Guido},
   title={Sulle metriche definite da una forme Hermitiana}, 
   language={Italian},
   journal={Atti del Reale Istituto Veneto di Scienze, Lettere ed Arti},
   date={1904},
   pages={502–513},
}

\bib{MR1288523}{book}{
   author={Griffiths, Phillip},
   author={Harris, Joseph},
   title={Principles of algebraic geometry},
   series={Wiley Classics Library},
   note={Reprint of the 1978 original},
   publisher={John Wiley \& Sons, Inc., New York},
   date={1994},
   pages={xiv+813},
   isbn={0-471-05059-8},
   review={\MR{1288523}},
   doi={10.1002/9781118032527},
}

\bib{MR0063670}{article}{
   author={Hirzebruch, Friedrich},
   title={On Steenrod's reduced powers, the index of inertia, and the Todd
   genus},
   journal={Proc. Nat. Acad. Sci. U. S. A.},
   volume={39},
   date={1953},
   pages={951--956},
   issn={0027-8424},
   review={\MR{0063670}},
   doi={10.1073/pnas.39.9.951},
}


\bib{MR0074086}{article}{
   author={Hirzebruch, Friedrich},
   title={Arithmetic genera and the theorem of Riemann-Roch for algebraic
   varieties},
   journal={Proc. Nat. Acad. Sci. U. S. A.},
   volume={40},
   date={1954},
   pages={110--114},
   issn={0027-8424},
   review={\MR{0074086}},
   doi={10.1073/pnas.40.2.110},
}

\bib{MR0003947}{book}{
   author={Hodge, W. V. D.},
   title={The Theory and Applications of Harmonic Integrals},
   publisher={Cambridge University Press, Cambridge, England; Macmillan
   Company, New York},
   date={1941},
   pages={ix+281},
   review={\MR{0003947}},
}

\bib{MR0161012}{book}{
   author={H\"{o}rmander, Lars},
   title={Linear partial differential operators},
   series={Die Grundlehren der mathematischen Wissenschaften, Bd. 116},
   publisher={Academic Press, Inc., Publishers, New York; Springer-Verlag,
   Berlin-G\"{o}ttingen-Heidelberg},
   date={1963},
   pages={vii+287},
   review={\MR{0161012}},
}

\bib{MR0172306}{article}{
   author={Kahn, P. J.},
   title={Characteristic numbers and oriented homotopy type},
   journal={Topology},
   volume={3},
   date={1965},
   pages={81--95},
   issn={0040-9383},
   review={\MR{0172306}},
   doi={10.1016/0040-9383(65)90071-6},
}

\bib{MR0165541}{article}{
   author={Kodaira, K.},
   title={On the structure of compact complex analytic surfaces. I, II},
   journal={Proc. Nat. Acad. Sci. U.S.A. 50 (1963), 218--221; ibid.},
   volume={51},
   date={1963},
   pages={1100--1104},
   issn={0027-8424},
   review={\MR{0165541}},
   doi={10.1073/pnas.51.6.1100},
}

\bib{MR0184257}{article}{
   author={Kodaira, K.},
   title={On compact analytic surfaces. II, III},
   journal={Ann. of Math. (2) 77 (1963), 563--626; ibid.},
   volume={78},
   date={1963},
   pages={1--40},
   issn={0003-486X},
   review={\MR{0184257}},
   doi={10.2307/1970500},
}

\bib{MR2815730}{article}{
   author={Malmendier, Andreas},
   title={The signature of the Seiberg-Witten surface},
   conference={
      title={Surveys in differential geometry. Volume XV. Perspectives in
      mathematics and physics},
   },
   book={
      series={Surv. Differ. Geom.},
      volume={15},
      publisher={Int. Press, Somerville, MA},
   },
   date={2011},
   pages={255--277},
   review={\MR{2815730}},
   doi={10.4310/SDG.2010.v15.n1.a8},
}

\bib{MR0082103}{article}{
   author={Milnor, John},
   title={On manifolds homeomorphic to the $7$-sphere},
   journal={Ann. of Math. (2)},
   volume={64},
   date={1956},
   pages={399--405},
   issn={0003-486X},
   review={\MR{0082103}},
   doi={10.2307/1969983},
}

\bib{MR0077122}{article}{
   author={Milnor, John},
   title={Construction of universal bundles. I},
   journal={Ann. of Math. (2)},
   volume={63},
   date={1956},
   pages={272--284},
   issn={0003-486X},
   review={\MR{0077122}},
   doi={10.2307/1969609},
}

\bib{MR0077932}{article}{
   author={Milnor, John},
   title={Construction of universal bundles. II},
   journal={Ann. of Math. (2)},
   volume={63},
   date={1956},
   pages={430--436},
   issn={0003-486X},
   review={\MR{0077932}},
   doi={10.2307/1970012},
}

\bib{MR0440554}{book}{
   author={Milnor, John W.},
   author={Stasheff, James D.},
   title={Characteristic classes},
   note={Annals of Mathematics Studies, No. 76},
   publisher={Princeton University Press, Princeton, N. J.; University of
   Tokyo Press, Tokyo},
   date={1974},
   pages={vii+331},
   review={\MR{0440554}},
}

\bib{MR0179172}{article}{
   author={N\'{e}ron, Andr\'{e}},
   title={Mod\`eles minimaux des vari\'{e}t\'{e}s ab\'{e}liennes sur les corps locaux et
   globaux},
   language={French},
   journal={Inst. Hautes \'{E}tudes Sci. Publ.Math. No.},
   volume={21},
   date={1964},
   pages={128},
   issn={0073-8301},
   review={\MR{0179172}},
   doi={10.1007/bf02684271},
}	

\bib{MR0180989}{article}{
   author={Novikov, S. P.},
   title={The homotopy and topological invariance of certain rational
   Pontryagin classes},
   language={Russian},
   journal={Dokl. Akad. Nauk SSSR},
   volume={162},
   date={1965},
   pages={1248--1251},
   issn={0002-3264},
   review={\MR{0180989}},
}

\bib{MR0193644}{article}{
   author={Novikov, S. P.},
   title={Topological invariance of rational classes of Pontryagin},
   language={Russian},
   journal={Dokl. Akad. Nauk SSSR},
   volume={163},
   date={1965},
   pages={298--300},
   issn={0002-3264},
   review={\MR{0193644}},
}

\bib{MR1104782}{article}{
   author={Oguiso, Keiji},
   author={Shioda, Tetsuji},
   title={The Mordell-Weil lattice of a rational elliptic surface},
   journal={Comment. Math. Univ. St. Paul.},
   volume={40},
   date={1991},
   number={1},
   pages={83--99},
   issn={0010-258X},
   review={\MR{1104782}},
}

\bib{MR0198494}{book}{
   author={Palais, Richard S.},
   title={Seminar on the Atiyah-Singer index theorem},
   series={With contributions by M. F. Atiyah, A. Borel, E. E. Floyd, R. T.
   Seeley, W. Shih and R. Solovay. Annals of Mathematics Studies, No. 57},
   publisher={Princeton University Press, Princeton, N.J.},
   date={1965},
   pages={x+366},
   review={\MR{0198494}},
}

\bib{MR0022667}{article}{
   author={Pontryagin, L. S.},
   title={Characteristic cycles on differentiable manifolds},
   language={Russian},
   journal={Mat. Sbornik N. S.},
   volume={21(63)},
   date={1947},
   pages={233--284},
   review={\MR{0022667}},
}

\bib{MR783704}{article}{
   author={Kvillen, D.},
   title={Determinants of Cauchy-Riemann operators on Riemann surfaces},
   language={Russian},
   journal={Funktsional. Anal. i Prilozhen.},
   volume={19},
   date={1985},
   number={1},
   pages={37--41, 96},
   issn={0374-1990},
   review={\MR{783704}},
}

\bib{MR0383463}{article}{
   author={Ray, D. B.},
   author={Singer, I. M.},
   title={Analytic torsion for complex manifolds},
   journal={Ann. of Math. (2)},
   volume={98},
   date={1973},
   pages={154--177},
   issn={0003-486X},
   review={\MR{0383463}},
   doi={10.2307/1970909},
}


\bib{MR0344216}{book}{
   author={Serre, J.-P.},
   title={A course in arithmetic},
   note={Translated from the French;
   Graduate Texts in Mathematics, No. 7},
   publisher={Springer-Verlag, New York-Heidelberg},
   date={1973},
   pages={viii+115},
   review={\MR{0344216}},
}

\bib{MR0059548}{article}{
   author={Serre, Jean-Pierre},
   title={Groupes d'homotopie et classes de groupes ab\'{e}liens},
   language={French},
   journal={Ann. of Math. (2)},
   volume={58},
   date={1953},
   pages={258--294},
   issn={0003-486X},
   review={\MR{0059548}},
   doi={10.2307/1969789},
}

\bib{MR1027535}{article}{
   author={Seeley, R.},
   author={Singer, I. M.},
   title={Extending $\overline\partial$ to singular Riemann surfaces},
   journal={J. Geom. Phys.},
   volume={5},
   date={1988},
   number={1},
   pages={121--136},
   issn={0393-0440},
   review={\MR{1027535}},
   doi={10.1016/0393-0440(88)90015-0},
}

\bib{MR584462}{article}{
   author={Stiller, Peter F.},
   title={Elliptic curves over function fields and the Picard number},
   journal={Amer. J. Math.},
   volume={102},
   date={1980},
   number={4},
   pages={565--593},
   issn={0002-9327},
   review={\MR{584462}},
   doi={10.2307/2374089},
}

\bib{MR0248858}{book}{
   author={Stong, Robert E.},
   title={Notes on cobordism theory},
   series={Mathematical notes},
   publisher={Princeton University Press, Princeton, N.J.; University of
   Tokyo Press, Tokyo},
   date={1968},
   pages={v+354+lvi},
   review={\MR{0248858}},
}

\bib{MR1511309}{article}{
   author={Study, E.},
   title={K\"{u}rzeste Wege im komplexen Gebiet},
   language={German},
   journal={Math. Ann.},
   volume={60},
   date={1905},
   number={3},
   pages={321--378},
   issn={0025-5831},
   review={\MR{1511309}},
   doi={10.1007/BF01457616},
}

\bib{MR0054960}{article}{
   author={Thom, Ren\'{e}},
   title={Espaces fibr\'{e}s en sph\`eres et carr\'{e}s de Steenrod},
   language={French},
   journal={Ann. Sci. Ecole Norm. Sup. (3)},
   volume={69},
   date={1952},
   pages={109--182},
   issn={0012-9593},
   review={\MR{0054960}},
}

\bib{MR0061823}{article}{
   author={Thom, Ren\'{e}},
   title={Quelques propri\'{e}t\'{e}s globales des vari\'{e}t\'{e}s diff\'{e}rentiables},
   language={French},
   journal={Comment. Math. Helv.},
   volume={28},
   date={1954},
   pages={17--86},
   issn={0010-2571},
   review={\MR{0061823}},
   doi={10.1007/BF02566923},
}

\bib{MR1403744}{article}{
   author={Vafa, Cumrun},
   title={Evidence for $F$-theory},
   journal={Nuclear Phys. B},
   volume={469},
   date={1996},
   number={3},
   pages={403--415},
   issn={0550-3213},
   review={\MR{1403744}},
   doi={10.1016/0550-3213(96)00172-1},
}

\bib{MR0010274}{article}{
   author={Whitney, Hassler},
   title={The self-intersections of a smooth $n$-manifold in $2n$-space},
   journal={Ann. of Math. (2)},
   volume={45},
   date={1944},
   pages={220--246},
   issn={0003-486X},
   review={\MR{0010274}},
   doi={10.2307/1969265},
}
\end{thebibliography}

\end{document}